\documentclass{article}

\usepackage[english]{babel}

\usepackage[letterpaper,top=2cm,bottom=2cm,left=3cm,right=3cm,marginparwidth=1.75cm]{geometry}

\usepackage{amsmath}
\usepackage{amsthm}
\usepackage{amsfonts}
\usepackage{amssymb}
\usepackage{graphicx}
\usepackage[colorlinks=true, allcolors=blue]{hyperref}
\usepackage{pdfpages}
\usepackage{comment}
\usepackage{dsfont}
\usepackage{ulem}
\usepackage{caption,subcaption}
\usepackage[export]{adjustbox}

\usepackage{authblk}

\newcommand{\one}{\mathds{1}}

\theoremstyle{plain}
\newtheorem{Th}{Theorem}[section]
\newtheorem{Lem}[Th]{Lemma}

\newtheorem{Prop}[Th]{Proposition}

\newcommand{\E}{\mathbb{E}}
\newcommand{\N}{\mathbb{N}}
\newcommand{\R}{\mathbb{R}}

\renewcommand{\P}{\mathbb{P}}
\newcommand{\RR}{\mathcal{R}}
\newcommand{\A}{\mathcal{A}}
\newcommand{\T}{\mathcal{T}}
\newcommand{\EE}{\mathcal{E}}
\newcommand{\Sb}{\overline{S}}
\newcommand{\Su}{\underline{S}}

\author[*,1]{Céline Bonnet}
\author[*,2]{Hélène Leman}
\affil[*]{ENSL, UMPA, CNRS UMR 5669, 69364 Lyon, France.}
\affil[1]{Email: celine.bonnet@ens-lyon.fr}
\affil[2]{Inria. Email: helene.leman@inria.fr}

\title{Site frequency spectrum of a rescued population under rare resistant mutations}
\date{}

\begin{document}
\maketitle

\begin{abstract}
    The aim of this article is to study the impact of resistance acquisition on the distribution of neutral mutations in a cell population under therapeutic pressure. The cell population is modeled by a bi-type branching process. Initially, the cells all carry type 0, associated with a negative growth rate. Mutations towards type $1$ are assumed to be rare and random, and lead to the survival of cells under treatment, i.e. type $1$ is associated with a positive growth rate, and thus models the acquisition of a resistance. Cells also carry neutral mutations, acquired at birth and accumulated by inheritance, that do not affect their type. We describe the expectation of the "Site Frequency Spectrum" (SFS), which is an index of neutral mutation distribution in a population, under the asymptotic of rare events of resistance acquisition and of large initial population. Precisely, we give asymptotically-equivalent expressions of the expected number of neutral mutations shared by both a small and a large number of cells. To identify the influence of relatives on the SFS, our work also lead us to study in detail subcritical binary Galton-Watson trees, where each leaf is marked with a small probability. As a by-product of this study, we thus provide the law of the generation of a randomly chosen leaf in such a Galton-Watson tree conditioned on the number of marks.
\end{abstract}

\noindent
\textbf{keywords:} Site Frequency Spectrum, rescue dynamics, sub and super-critical branching processes, Galton-Watson trees.

\medskip

\noindent
\textbf{MSC classes:} 60F05, 60J28, 60J80, 92D15.

\medskip

\section{Introduction}
\label{sec:intro}
The Site Frequency Spectrum (SFS) is a statistical object that records the distributions of some mutations in an evolving population along time and it has provided a simple means of understanding the evolutionary history of populations from genomic data (\cite{Gunnarssonetal2021}). Considering the increasing amount of genomic data collected, the SFS has thus become a key object to study. It has been studied in many forms (on sampled or unsampled populations, as a limit for large sample size, as a long time limit,...) with different models (Wright-Fisher, Moran, coalescent, semi-deterministic, birth and death process, Galton-Watson process,...). Recently, Dinh et al, in \cite{Dinhetal2020}, compared and discussed some of these different approaches using simulations with some mathematical analyses.

In our article, we are interested in understanding the impact of a rescue dynamics on the SFS associated to the accumulation of neutral mutation in a population of two types of individuals.
Previous works have studied distribution of mutations, produced by two-type (or more) branching processes as in \cite{AntalandKrapivsky2011}, \cite{Bozicetal2019}, \cite{Bravermanetal95}, \cite{CheekandAntal2018}, \cite{CheekandAntal2020}, \cite{Durrett2013}, \cite{Dinhetal2020}, \cite{TungandDurrett2021}, \cite{KurpasandKimmel2022}, among others. In those studies as in our article, each population is characterized by a growth rate, representing a type, and non-neutral mutations are non-reversible. Moreover, the individuals accumulate neutral mutations by inheritance and during their lifetime, whether at birth as in our case, or throughout life. 

However, a fundamental assumption regarding the growth of the initial population distinguishes our study from the previously cited works. Indeed, in these previous studies, all populations or at least initial populations  are assumed to be supercritical (or critical), i.e. the growth rates of individuals are positive (or null). In our work, the initial population is assumed to have a large size and is described by a sub-critical birth and death process (i.e. its size decreases exponentially fast). The second population, resulting by rare mutations from the first one, follows a supercritical process. Such dynamics are commonly called \textit{rescue} dynamics.

Studying \textit{rescue} dynamics is particularly important in the context of oncology. Indeed treatments, especially chemotherapy, exert significant selection pressures on cell populations and can thus favor the emergence of resistant populations. Our process aims at modeling this specific mechanism. Precisely, the first population models a cancer cells population sensitive to a given treatment. While the treatment is administered, cells can become resistant to it and form a second population which can increase even under the treatment (and thus can be rescued). This idea is justified by some evidence in oncology. For example, in \cite{Ollieretal2017}, Ollier et al studied the resistance to a chemotherapy, named temozolomide, in low-grade gliomas. Using longitudinal tumor size measurements and mathematical models, they show the correlation between this chemotherapy and the development of resistance for the patient in half of the cases.
Precisely, in half of the cases, a mathematical model considering the emergence of resistant cells while the treatment is given fits the data better than a model assuming that resistant cells are present before treatment.  \\

We will study the rescue dynamics in a multi-scale context. We will assume that initially there is no resistant cell and there is a large number $N$ of sensitive cells ($N>>1$). Each sensitive cell can become resistant at birth  with probability proportional to $\frac{1}{N^{\alpha}} $, with $0<\alpha \leq 1$. Such an assumption is commonly called a rare mutation assumption since the probability decreases to $0$ with $N$. In the same spirit, Durrett and Schweinsberg in \cite{DurrettandSchweinsberg2004} study the distribution of two specific mutations in a cell population and assume a recombination probability depending on the population size. In \cite{Bertoin2010}, such an assumption is made on neutral mutations' occurrence. Cheek and Antal in \cite{CheekandAntal2018} and \cite{CheekandAntal2020} also include such multi-scale assumption but they restricted their attention to the case where the expected number of mutational events is finite. On the contrary, in our study, we include both cases where the expected number of rescue events is finite or infinite, as this quantity is of order $N^{1-\alpha}$ with $\alpha\in (0,1]$. In the following, the resistant cells born from a rescue event will be called \textit{ancestral resistant} cells.\\

The SFS classifies mutations according to the number of cells that carry them. In order to describe the SFS in such a rescue dynamics context, we need to identify mutations carry by more than one \textit{ancestral resistant} cell. To this end, we show some general properties of the Galton-Watson tree associated with each initial sensitive cell. Previous works have studied rescue dynamics using multi-type branching processes as \cite{AzevedoandOlofsson2021}, \cite{DurrettandMoseley2010}, \cite{Iwasaetal2006}, \cite{Kendall1960}, \cite{Komarova2006}, \cite{FooandLeder2013}, \cite{OrrandUnckless2014}, among others. To our knowledge, none of them have been interested in describing the relationship between the \textit{resistant ancestral} cells. In \cite{SagitovandSerra2009}, authors use the Galton-Watson structure to capture times of rescue events due to an accumulation of mutations without looking at the relationship between the \textit{ancestral resistant} cells.\\

The paper is organized as follows. In Section \ref{sec:model}, we introduce the model, the multi-scale assumptions, and the definition of the SFS we are looking for. Then we present the main results. Section~\ref{sec:AR} is devoted to the study of the \textit{ancestral resistant} cells dynamics, which are resistant cells whose mother are sensitive cells. These cells play a key roll in our proofs, as we will split our quantities of interest, related to the SFS, into the contributions of neutral mutations that appeared before (resp. after) the occurrence of \textit{ancestral resistant} cells. In particular, we will detail in Section~\ref{sec:AR}, the law of occurrence of \textit{ancestral resistant} cells and the law of the number of neutral mutations they carry.
The proof of the main theorems are presented in Section~\ref{sec:SFS}. We start by studying the contributions of neutral mutations that appeared after the occurrence of \textit{ancestral resistant} cells in the subsection~\ref{ssec:sfsR}. Then we deal in Subsection~\ref{ssec:sfsS} with the contribution of these emerging after the appearance of \textit{ancestral resistant} cells. Finally, Section~\ref{sec:illus} include numerical illustration of the results and some discussions about these results.

\section{Model description}\label{sec:model}
We consider a population of cancer cells with cells sensitive to a treatment and cells resistant to it. Our model describes the number of sensitive cells $Z^N_0(t)$ at time $t\geq 0$, the one of resistant cells $Z^N_1(t)$ at time $t\geq 0$, and the number of \textit{neutral mutations} carried by each cell. A \textit{neutral mutation} refers to a mutation that has no impact on the birth and death rates of the cell under the treatment/medium in which it is observed. The initial state of the process is assumed to be 
\begin{equation*}
(Z^N_0(0),Z^N_1(0))=(N,0).
\end{equation*}
Cells carry no neutral mutation initially, and the process follows the following dynamics:
\begin{itemize}
    \item Each sensitive cell divides at rate $b_0$ and dies at rate $d_0$. We denote by $\lambda_0:=d_0-b_0> 0$ the absolute value of the growth rate of sensitive cells.
    \item Each resistant cell divides at rate $b_1$ and dies at rate $d_1$. Its growth rate, which is positive, is denoted by $\lambda_1:=b_1-d_1> 0$.
    \item At each division, the cell is replaced by two daughter cells:
    \begin{itemize}
        \item Each daughter cell inherits the neutral mutations of their mother in addition to an independent random number of neutral mutations $N_\omega$ with the following expectation:
        $$
         \E[N_{\omega}]=\omega/2\geq 0.$$
        \item Each of the two daughter cells may become resistant with probability $\gamma_N:=\gamma/N^{\alpha}$, with $0<\alpha\leq 1$, independently from one another. The parameter $\alpha$ models the rarity of the occurrence of resistances.
    \end{itemize} 
\end{itemize}
Hence notice that, considering all the dynamics, the (induced) growth rate of sensitive cells is $-(\lambda_0+2\gamma_Nb_0)$, which is negative.\\

 Finally, the dynamics of the process $Z^N(t)=(Z^N_0(t),Z^N_1(t))_{t\geq 0}$ (disregarding neutral mutations) is the one of a continuous-time Markov chain on $\N^2$ with the following rates: 
\begin{center}
\begin{tabular}{lllll}
    $(z_0,z_1)$ & $\mapsto$ & $(z_0+1,z_1)$ & at rate & $(1-\gamma_N)^2b_0z_0$ \\
      $(z_0,z_1)$ & $\mapsto$ & $(z_0-1,z_1)$ & at rate & $d_0z_0$ \\
       $(z_0,z_1)$ & $\mapsto$ & $(z_0,z_1+1)$ & at rate & $2\gamma_N(1-\gamma_N)b_0z_0+b_1z_1$ \\
       $(z_0,z_1)$ & $\mapsto$ & $(z_0-1,z_1+2)$ & at rate & $\gamma_N^2b_0z_0$ \\
       $(z_0,z_1)$ & $\mapsto$ & $(z_0,z_1-1)$ & at rate & $d_1z_1$ \\
\end{tabular}
\end{center}

Our aim is to describe the distribution of neutral mutations in the resistant cells population at a large time, precisely after the characteristic time of extinction of sensitive cells, i.e. $\log(N)/\lambda_0$. As previously mentioned, we will describe this distribution using the \textit{site frequency spectrum} (SFS) which counts, for all $i\in \N^*$ and $t\geq 0$, the number of neutral mutations carried by exactly $i$ resistant cells at time $t$. This sequence will be denoted by $(S_i(t))_{i\in \N, t\geq 0}$.

For clarity, we set an example starting from one sensitive cell on Figure \ref{figTree}. The SFS associated to this progeny is given by $S_1(t)=3$, $S_3(t)=1$, $S_7(t)=2$ and for all $i\notin \{1,3,7\}$, $S_i(t)=0$. Indeed, in this example, only $7$ resistant cells are alive at time $t$. All resistant cells at time $t$ carried mutations $\{1\}$ and $\{2\}$. Mutation $\{4\}$ is carried by three cells while mutations $\{7\}$, $\{8\}$ and $\{9\}$ are carried by only one resistant cell. We are only interested in resistant cells alive at time $t$, so mutations $\{3,5,6,10\}$ which are carried by only sensitive cells don't impact the SFS we are interested in.

\begin{figure}[ht]
  \centering
    \includegraphics[width=0.8\textwidth]{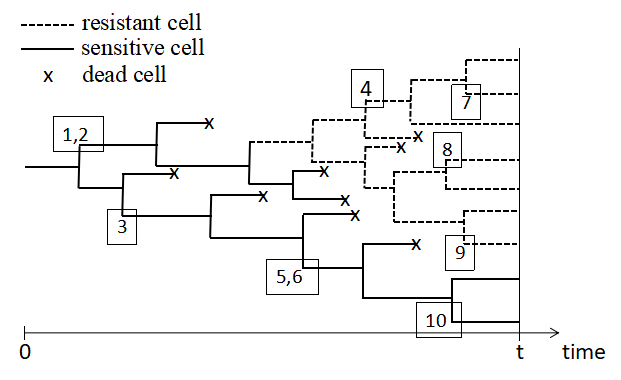}
  \caption{Example of progeny starting from one sensitive cell and including neutral mutations. Solid lines correspond to sensitive genealogy, dashed lines to resistant one, cross to dead cells and framed numbers to neutral mutations appeared at each division. }\label{figTree}
\end{figure}

Our aim is thus to describe the expectation of the SFS of the resistant population in a limit of large initial sensitive population, i.e. $N$ tends to $+\infty$, and in the timescale of the sensitive population extinction, i.e. we will set results for
$$
t_N:=t \log(N), \text{ with } t>0.
$$
Similarly to \cite{Dinhetal2020} and \cite{TungandDurrett2021}, $S_i^N(t)$ will be computed by separating the quantity into two different contributions: (1) $\Sb_i^N(t)$ that counts neutral mutations that appeared in a resistant cell 
and (2) $\Su_i^N(t)$ that counts neutral mutations that appeared in a sensitive cell, which are transmitted to resistant cells as a hitch-hiking effect (\cite{MaynardSmithandHaigh1974}).

The first result details the expected SFS for fixed $i$, i.e. it describes the number of mutations shared by a small number of resistant cells at time $t_N$. 
\begin{Th}\label{mainth}
For all $i\in \N$, $t\geq 0$,
\begin{equation}\label{def:I}
\E\left[S_i^N\big(t_N\big)\right] \underset{N\to \infty}{\sim} I(i)\frac{2b_0\gamma\omega}{\lambda_1+\lambda_0}N^{\lambda_1 t + 1-\alpha}, \text{ with } I(i):=\int_0^1\frac{1-y}{1-d_1y/b_1}y^{i-1}dy.
\end{equation}
\end{Th}
Notice how the dynamics of the sensitive cells affects the asymptotically-equivalent expression of the expected SFS. In \cite{Gunnarssonetal2021} and \cite{Richard}, such an equivalent was obtained for a birth and death process starting with one resistant cell (see Lemma~\ref{lemS1}). Moreover, Lemma~\ref{lem:nbrresistant} states the number of resistant events that occur during the process of extinction of the sensitive population. By noticing that the resistance events all appear in a negligible time relatively to the time of interest $t_N$, one may think that the final expectation of the SFS would correspond to the multiplication of the two asymptotically-equivalent expressions given by the two lemmas previously cited, i.e.
$$
I(i)\frac{2b_0\gamma\omega}{\lambda_0}N^{\lambda_1t+1-\alpha}.
$$
This is not the case. Indeed, notice that the growth rate of the resistant population modifies the constant parameter. Finally, we can conclude that the \textit{rescue} dynamics has a significant effect on the SFS, although the order size of the approximation ($N^{\lambda_1 t+1-\alpha}$) and its shape with respect to $i$ is not directly impacted.

As previously indicated, the proof of this theorem~\ref{mainth} will be done by studying two distinct quantities $\E[\Sb_i^N(t_N)]$ and $\E[\Su_i^N(t_N)]$ (cf Lemma~\ref{lem:Sbieq} and Lemma~\ref{lem:Sui} respectively). In this case where $i$ is fixed, notice that nearly all of the contribution is accounted for by the first quantity $\E[\Sb_i^N(t_N)]$, i.e. by the neutral mutations occurring in a resistant cell. Indeed, Lemma~\ref{lem:Sbieq} and~\ref{lem:Sui} imply that the second quantity is negligible with respect to the first one.\\

However, the approximation given by Theorem~\ref{mainth} is not appropriate to study the "large families", i.e. mutations that affect a large number of resistant cells.  Indeed, as $t_N$ increases with $N$, the order size of the total population at $t_N$ is $e^{\lambda_1 t_N}$ conditioned to the \textit{rescue}. We should thus study the SFS for some $i$ depending on $N$ as $i\sim N^{\lambda_1 t}$. Hence, we chose to study $S_{x_1,x_2}(t)$, with $x_1,x_2\in (0,\infty]$ and $x_1<x_2$, and which is the number of mutations carried by a number of resistant cells between $x_1\, e^{\lambda_1 t}$ and  $x_2 \, e^{\lambda_1 t}$ at time $t$. Let us now state the result associated to this quantity.

\begin{Th}\label{mainth2}For all $t\geq 0$, $x_1,x_2\in (0,\infty]$ with $x_1<x_2$, we denote 
\begin{equation}\label{def_Sx1x2}
S_{x_1,x_2}(t)=\sum_{i\in (x_1e^{\lambda_1 t_N},x_2e^{\lambda_1 t_N})} S_i(t)
\end{equation}
Then for all $x_1,x_2\in (0,\infty]$ with $x_1<x_2$, $t\geq 0$,
\begin{equation}
       \label{eq:expSxapp}
         \E\left[S_{x_1,x_2}^N(t_N)\right]\underset{N\to\infty}{\sim} \displaystyle b_0 \gamma \omega \lambda_1 \Big(J(x_1)- J(x_2)\Big) N^{1 -\alpha}. 
       \end{equation}
       where for all $x>0$, $J(x)=K(x) + L(x)$ with 
 \begin{align}
     K(x)&:= \frac{2}{\lambda_0+\lambda_1}\int_0^\infty (e^{-(\lambda_1+\lambda_0)s}-1)\, e^{\lambda_1 \,s}e^{-x\frac{\lambda_1}{b_1}e^{\lambda_1s}}ds,\label{def:K}\\
      L(x)&:= \frac{1}{b_1} \int_0^{+\infty}(1+2b_0s) \,e^{-\lambda_0 s}e^{-x\frac{\lambda_1}{b_1}e^{\lambda_1 s}} ds.\label{def:L}
 \end{align}
\end{Th}
As previously, this theorem will be proved by splitting the SFS into two contributions: $\E[\Sb_{x_1,x_2}^N(t_N)]$ describing the number of mutations that appeared during resistant divisions (cf Lemma~\ref{lem:Sbxeq}) and $\E[\Su_{x_1,x_2}^N(t_N)]$  describing the number of mutations that appeared during sensitive divisions (cf Lemma~\ref{lem:expSuiNapp}).  Contrary to Theorem~\ref{mainth}, $\E[\Sb_{x_1,x_2}^N(t_N)]$ and $\E[\Su_{x_1,x_2}^N(t_N)]$ both contribute to the asymptotically-equivalent expression of $\E[S_{x_1,x_2}^N(t_N)]$. The difference between their contributions lies in the shape of the functions $K$ and $L$, respectively associated to $\E[\Sb_{x_1,x_2}^N(t_N)]$ and $\E[\Su_{x_1,x_2}^N(t_N)]$. We deduce from Theorem~\ref{mainth2} that the rescue dynamics impact the asymptotic expected of the SFS of large families  thought three different ways:
\begin{itemize}
    \item The mean number of \textit{ancestral resistant} cells given in Lemma~\ref{lem:nbrresistant};
    \item The lost of time due to the growth rate of the sensitive cells given by $\lambda_0$;
    \item The number of times the sensitive cells divide before they become resistant. Such influence is described by the parenthesis $(1+2b_0s)$ in the expression of $L$. Surprisingly, we find a factor $2$, already met in the dynamics of branching processes (see remark of the main theorem in \cite{Bansaye}) which translates the increase of the probabilities to see a resistant cell appearing in a lineage having many divisions.
\end{itemize}

Let us now deal with the proof of the results.

\section{Description of \textit{ancestral resistant} cells}
\label{sec:AR}

Let us remind that we call \textit{ancestral resistant} cell a resistant cell whose mother is a sensitive cell.
In this section, we establish results regarding the law of occurrence of \textit{ancestral resistant} cells and the probabilities that the progeny of an initial sensitive cell carries one or more \textit{ancestral resistant} cells.\\

In our first result,  Proposition \ref{th:ancestralresistant}, we describe the law of the time of appearance of \textit{ancestral resistant} cells and the law of the number of neutral mutations they carry.

To this end, we will study the progeny of one sensitive cell alive at time $0$. We assume that $\alpha>0$, hence the probability for two or more \textit{ancestral resistant} cells to emerge in the  progeny of the same initial sensitive cell is insignificant compared to the probability of emergence of a unique \textit{ancestral resistant} cell in the progeny of an initial sensitive cell (see Lemma \ref{lem:probaanc} for rigorous arguments). So we will study the structure of progeny containing exactly one \textit{ancestral resistant} cell. Proposition \ref{th:ancestralresistant} gives the law of $G_N$ and $T_N$, respectively, the generation and the appearance time of an \textit{ancestral resistant} cell conditioned on belonging to a progeny that carried exactly one \textit{ancestral resistant} cell.

\begin{Prop}\label{th:ancestralresistant}
 For any $N\in \N^*$, let 
	\begin{equation}
	\label{def:pNbetaN}
	    p_N=\frac{(1-\gamma_N)b_0}{b_0+d_0}, \quad \beta_N=\displaystyle \frac{(b_0+d_0)\gamma_N}{d_0+b_0\gamma_N} 
	\quad \text{and} \quad x_N=\sqrt{1-4p_N(1-p_N)(1-\beta_N)}.
	\end{equation}
	
Then
	\begin{itemize}
		\item[(i)] the law of $G_N$ is characterized by, for all $g \in \mathbb{N^*}$, \begin{equation}
		\label{eqG}
		\mathbb{P}(G_N=g)= x_N (1-x_N)^{g-1}
		\end{equation}
	
		\item[(ii)] and the density of $T_N$, $f_{T_N}$, is written,  for all  $t\geq0$, 
		\begin{equation}\label{eqfT}
		f_{T_N}(t) = \delta_0 x_N e^{-t\delta_0 x_N},
		\end{equation}
  with \begin{equation}\delta_0:=b_0+d_0\label{def:delta0}.\end{equation}
	\end{itemize}

\end{Prop}

Using the proof of this proposition, we are able to state the law of one \textit{ancestral resistant} cell chosen uniformly at random and belonging to a progeny that carried at least one \textit{ancestral resistant} cell (and not exactly one as for Proposition~\ref{th:ancestralresistant}). Such a result is given by Lemma~\ref{lem:TN_annexe}. Moreover, we are also able to state the following lemma, which deals with the probability that a progeny carries a given number of \textit{ancestral resistant} cells and which will be useful to prove our main results.
\begin{Lem}\label{lem:probaanc}
Let denote by $\A_k^N$ the event that there is exactly $k$ \textit{ancestral resistant} cells in the progeny of one (initial) sensitive cell, then
\begin{equation}
    \label{eq:1anc}
    \P\left(\A_1^N\right)=\frac{1-x_N}{x_N(1-\gamma_N)}\gamma_N=\frac{2b_0}{\lambda_0}\gamma_N+\underset{N\to\infty}{o}(\gamma_N),
\end{equation}
where $x_N$ is defined in~\eqref{def:pNbetaN}.

Moreover, the expected number of \textbf{multiple} \textit{ancestral resistant} cells in the progeny of a (initial) sensitive cell is given by
\begin{equation}
    \label{eq:k2anc}
    \sum_{k\geq 2}k\P\left(\A_k^N\right)
    =\frac{2b_0\delta_0^2}{\lambda_0^3}\gamma_N^2+\underset{N\to\infty}{o}(\gamma_N^2). 
\end{equation}
\end{Lem}
All the results of this section are true for all $\gamma_N \in (0,1)$. Hence we obtain that for all $\gamma_N \in (0,1)$, $G_N$ follows a geometric law of parameter $x_N \in (0,1)$ given by Formula \eqref{def:x} and $T_N$ follows an exponential law of parameter $x_N\delta_0$.The parameter $\delta_0$ represents the global dynamical rate of one sensitive cell. The parameter $x_N$ can be seen as the probability for a sensitive cell to become resistant knowing $\A_1^N$, i.e. knowing there is only one \textit{ancestral resistant} cell in the progeny of its initial ancestor. Hence to have one resistant ancestral cell at generation $g$, $g-1$ divisions need to be done without events of resistance and an event of resistance has to succeed at the $g^{\text{th}}$ attempt. We illustrate the results of Proposition \ref{th:ancestralresistant} with numerical simulations in Figure \ref{fig:Law}. Orange histograms have been obtained with a sample of $100 \,000$ realizations of $G_N$ in cases (a) and (b), and of $100 \,000$ realizations of $T_N$ for cases (c) and (d). Two different values of $\gamma_N$ were used $\gamma_N=0.2$ for (a) and (c) and $\gamma_N=0.002$ for (b) and (d). We can see that the bigger $\gamma_N$ is, the more realizations of $G_N$ are close to $1$ and of $T_N$ close to $0$. Indeed, the process $Z_0$ is subcritical. Hence the probability to observe old sensitive cells is very rare. When $\gamma_N$ is small, we need to simulate a lot of events to observe $100\,000$ realizations of $G_N$. So we increase the chance of observing rare events as old sensitive cells.\\

\begin{figure}[ht]
  \centering
  \subcaptionbox{\centering Empirical and theoretical law of $G_N$ for $\gamma_N=0.2$.\label{figGN1}}{%
    \includegraphics[width=0.5\textwidth]{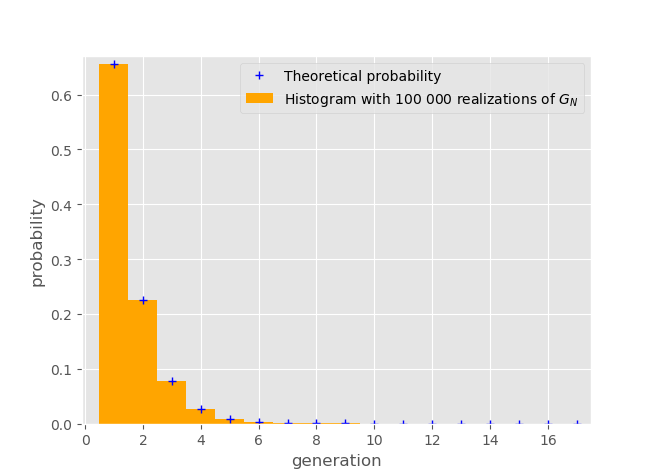}}\subcaptionbox{\centering Empirical and theoretical law of $G_N$ for $\gamma_N=0.002$.\label{figGN2}}{%
    \includegraphics[width=0.5\textwidth]{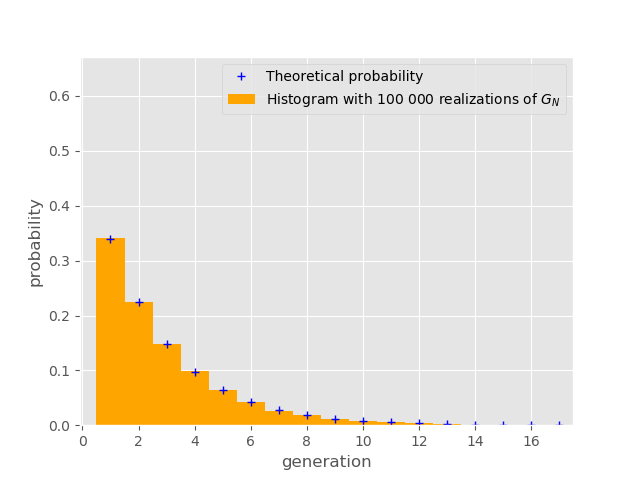}}
  \subcaptionbox{\centering Empirical and theoretical law of $T_N$ for $\gamma_N=0.2$.\label{figTN1}}{%
    \includegraphics[width=0.5\textwidth]{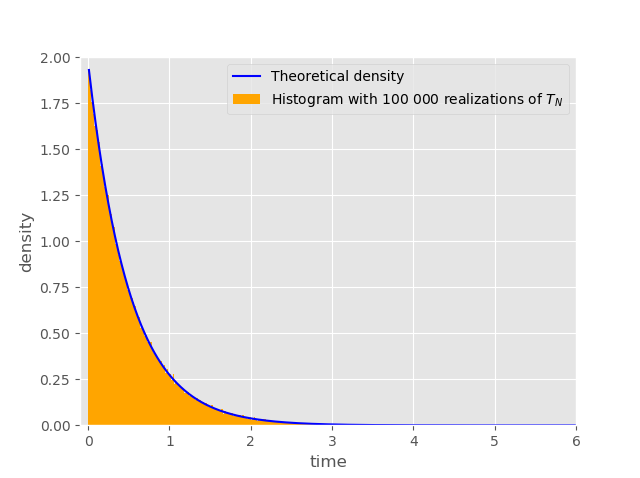}}\subcaptionbox{\centering Empirical and theoretical law of $T_N$ for $\gamma_N=0.002$.\label{figTN2}}{%
    \includegraphics[width=0.5\textwidth]{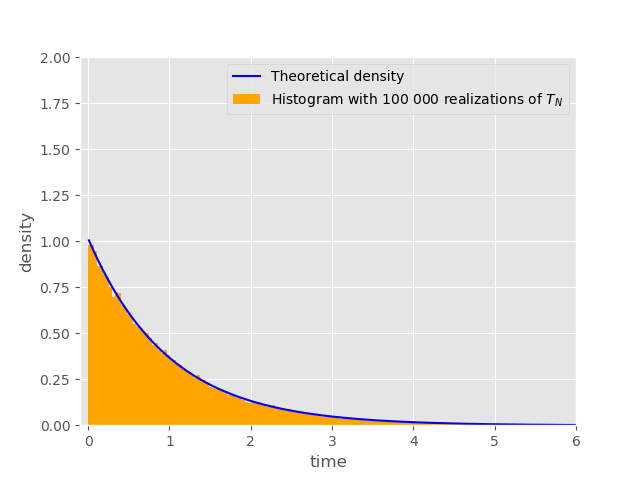}}
  \caption{Theoretical and empirical law of $T_N$ and $G_N$ for $b_0=1.0$, $d_0=2.0$ and two different values of $\gamma_N$ ($\gamma_N=0.2$ at left and $\gamma_N=0.002$ at right). The orange histograms has been obtained using $100 \, 000$ realizations of $G_N$ (for Figures (a) and (b)) and of $T_N$ (for Figures (c) and (d)). The blue cross of Figures (a) and (b) correspond to Formula \eqref{eqG}. The blue lines of Figures (c) and (d) correspond to the function \eqref{eqfT}.}\label{fig:Law}
\end{figure}

The end of this section is devoted to the proofs.

\begin{proof}[Proof of Proposition~\ref{th:ancestralresistant}]
    The main idea to prove this  proposition is to study the topology of the subcritical family trees issued from the initial sensitive cells, disregarding the neutral mutations and the offspring of \textit{ancestral resistant} cells. To this aim, we will slightly modify the trees as follows:
    \begin{itemize}
        \item Firstly, as we are first interested in the generation of the \textit{ancestral resistant} cells, we consider fixed time of living. In other words, we consider Galton-Watson trees. 
        \item At the end of the living time of a cell (except the one at the tree root), it becomes a resistant cell with probability $\gamma_N$. In other words, we disconnect the event of division and the event of becoming resistant, such that in the new considered tree, cells divide and at the end of their living time, we decide if they were actually resistant or not (with probability $\gamma_N$). 
        \item If a cell become resistant at the end of its living time, this leaves a leaf, as if it was dead. Indeed, we are not interested in the offspring of the resistant cells to prove this theorem. 
    \end{itemize}
    To summarize, the trees under consideration are the following ones. At the end of its living time, a cell can
    \begin{itemize}
        \item[-] become resistant with probability $\gamma_N$ and have no progeny;
        \item[-] die with probability $(1-\gamma_N)\frac{d_0}{b_0+d_0}$;
        \item[-] divide into two with probability $(1-\gamma_N)\frac{b_0}{b_0+d_0}=:p_N$.
    \end{itemize}
    Notice also that we will have to exclude the possibility that the root becomes resistant at the end of its life time, since in our initial process the roots are all sensitive cells.\\
     On Figure \ref{fig:modif}, we give an example of tree with its corresponding modified tree. The probability of each event is given above the branch. We can notice that the probability of these two configurations of tree is actually the same.\\
    \begin{figure}[h]
  \centering
  \subcaptionbox{Initial tree.\label{fig1a}}{%
    \includegraphics[width=0.4\textwidth]{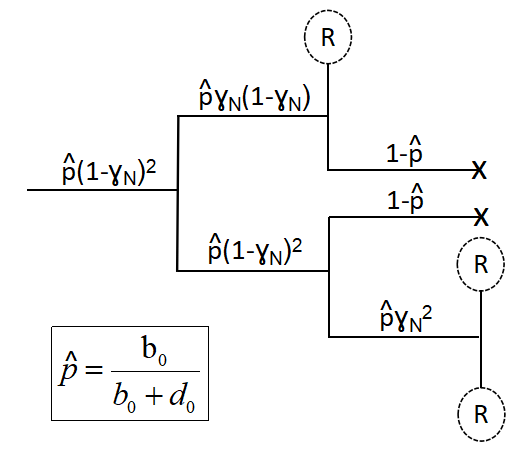}}
  \subcaptionbox{Corresponding modified tree.\label{fig1b}}{%
    \includegraphics[width=0.5\textwidth]{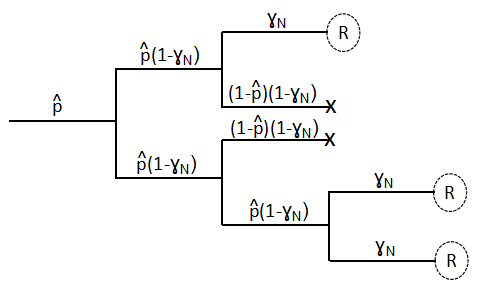}}
 
  \caption{An example of family tree (a) without modification or (b)with the modification. The probability of each event, leading to this tree, is given up to each branch. In the modified case (b), the probability for the root to divide is given by $\hat{p}$ which correspond to division event knowing that the root can't become resistant.}\label{fig:modif}
\end{figure}
    
    To proceed with the proof of the proposition, we need the following general lemma concerning Galton-Watson trees that we state here but prove at the end of this section.
    \begin{Lem}\label{lemma:GWgen}
    Let $\mathcal{T}(p,\beta)$ be a subcritical Galton-Watson tree such that branches divide into two with probability $p< 1/2$ or die with probability $1-p$ and such that, once the (finite) tree is constructed, each leaf is marked independently from one another with probability $\beta$. Let denote by $\mathcal{R}^{\mathcal{T}(p,\beta)}$ the event that $\mathcal{T}(p,\beta)$ has exactly one marked leaf. Finally, let $G^{\mathcal{T}(p,\beta)}$ be the generation of a leaf chosen uniformly at random in a tree of law $\mathcal{T}(p,\beta)$. The generation of a leaf is the smallest number of edges on a path between the leaf and the root. Then, for all $g\in \mathbb{N}^*$,
	\begin{equation}
	\label{eqPhr}
	\P\left(G^{\mathcal{T}(p,\beta)}=g \; \vert \; \mathcal{R}^{\mathcal{T}(p,\beta)}\right)= x(1-x )^{g-1}
	\end{equation} where \begin{equation}\label{def:x}
	x=\sqrt{1-4p(1-p)(1-\beta)}.
	\end{equation}
	\end{Lem}
	
 Considering the notation of the previous Lemma~\ref{lemma:GWgen}, the tree under consideration is $\T_N:=\mathcal{T}(p_N, \beta_N)$, with $p_N$ and $\beta_N$ defined in \eqref{def:pNbetaN}, conditioned on the event $\mathcal{E}^{\T_N}$ that the root can not become resistant, and identifying marked leaves with \textit{ancestral resistant} cells. Indeed, according to the previous consideration, $p_N$ corresponds to the probability that a cell divides and
	$$
	\beta_N=\frac{\gamma_N(b_0+d_0)}{d_0+\gamma_N b_0}=\frac{\gamma_N}{1-p_N}
	$$
    corresponds to the probability that a leaf is the result of a "resistant event" and not a "death event".
	Hence, we have for all $g\in \N^*$, 
	\begin{equation*}
	    \P(G_N=g)=\P\left(G^{\T_N}=g+1|\mathcal{R}^{\T_N},\mathcal{E}^{\T_N}\right)
	\end{equation*}
	On the event $\mathcal{R}^{\T_N}\cap\mathcal{E}^{\T_N}$, the root of $\T_N$ divides into two cells with probability $1$ that will then evolve independently from one another and that will give two independent subtrees $\T_N^{(1)}$ and $\T_N^{(2)}$ whose laws follow the same law as $\T_N$. Let us denote by $\RR^{\T}_0$ the probability that the tree $\T$ has no marked leaf. Thus, using the Markov property
	\begin{align*}
	    \P(G_N=g)&=\frac{\P\left(\left\{G^{\T_N}=g+1\right\}\cap\mathcal{R}^{\T_N}\cap\mathcal{E}^{\T_N}\right)}{\P\left(\mathcal{R}^{\T_N}\cap\mathcal{E}^{\T_N}\right)}\\
	    &=\frac{p_N\left[\P\left(\left\{G^{\T_{N}^{(1)}}=g\right\}\cap\mathcal{R}^{\T_{N}^{(1)}}\cap\mathcal{R}_0^{\T_{N}^{(2)}}\right)+\P\left(\left\{G^{\T_{N}^{(2)}}=g\right\}\cap\mathcal{R}^{\T_{N}^{(2)}}\cap\mathcal{R}_0^{\T_{N}^{(1)}}\right)\right]}{p_N\left[\P\left(\mathcal{R}^{\T_{N}^{(1)}}\cap\mathcal{R}_0^{\T_{N}^{(2)}}\right)+\P\left(\mathcal{R}^{\T_{N}^{(2)}}\cap\mathcal{R}_0^{\T_{N}^{(1)}}\right)\right]}.
	\end{align*}
	Since $\T_N^{(1)}$ and $\T_N^{(2)}$ are two independent trees with the same law as $\T_N$, we finally have,
	\begin{align*}
	    \P(G_N=g) =\frac{2\P\left(\left\{G^{\T_N}=g\right\}\cap\mathcal{R}^{\T_N}\right)\P\left(\mathcal{R}_0^{\T_N}\right)}{2\P\left(\mathcal{R}^{\T_N}\right)\P\left(\mathcal{R}_0^{\T_N}\right)}.
	\end{align*}
	We conclude the proof of \eqref{eqG} using Lemma \ref{lemma:GWgen}.\\
	
	To find the density of $T_N$, it is sufficient to notice that the life time of each sensitive cells is distributed as exponential r.v. with parameter $\delta_0=b_0+d_0$, and that an \textit{ancestral resistant cell} has $G_N$ ancestors. Thus, $T_N=\sum_{i=1}^{G_N}\mathcal{E}_i$, where $(\mathcal{E}_i)_{i\in \N}$ is a sequence of i.i.d. r.v. of exponential law with parameter $\delta_0$ and independent from $G_N$. Hence $T_N$ follows an exponential law with parameter $x_N\delta_0$. This ends the proof of Proposition~\ref{th:ancestralresistant}.
	
\end{proof}

\begin{proof}[Proof of Lemma~\ref{lemma:GWgen}]
    
    To simplify notation in the proof of this lemma, we will write $\mathcal{T}$ for $\mathcal{T}(p,\beta)$, $\RR$ for $\RR^\T$ and $G$ for $G^\T$. 
    The main idea to prove Equation~\eqref{eqPhr} is to condition on the number of leaves. We will denote the number of leaves of tree $\T$ by $\lambda(\T)$. Then, for all $g\in \N^*$,
    \begin{align*}
        \P(G=g|\RR)&=\frac{\sum_{n\geq 1}\P(G=g,\RR,\lambda(\T)=n)}{\P(\RR)}\\
        &=\frac{\sum_{n\geq 1}\P(\RR|G=g,\lambda(\T)=n)\P(G=g,\lambda(\T)=n)}{\sum_{n\geq 1}\P(\RR|\lambda(\T)=n)\P(\lambda(\T)=n)}.
    \end{align*}
    Notice that for all $g,n\in \N^*$, $\P(\RR|G=g,\lambda(\T)=n)=\P(\RR|\lambda(\T)=n)=n\beta (1-\beta)^{n-1}$. Indeed, the probability to have exactly one resistant among the n leaves of $\T$ is given by $\beta (1-\beta)^{n-1}$ and we have $n$ possible configurations. Then by writing for all $g,n\in \N^*$, 
    \begin{equation}
        \label{def:unvhn}
        v_{g,n}:=\P(G=g,\lambda(\T)=n)\quad \text{and} \quad u_n:=\P(\lambda(\T)=n),
    \end{equation}
    we find
    \begin{equation}\label{eq:HcondR1}
        \P(G=g|\RR)=\frac{\sum_{n\geq 1}n(1-\beta)^{n-1}v_{g,n}}{\sum_{n\geq 1}n(1-\beta)^{n-1}u_n}.
    \end{equation}
    To end the proof, we thus have to study the two sums of the r.h.s. of Equation~\eqref{eq:HcondR1}.

    Let us first deal with $\sum_{n\geq 1}n(1-\beta)^{n-1}u_n$. To this aim, notice that using the Markov property at the time of the first generation, we have for all $n\in\N^*$
    \[
    u_n=\one_{\{n=1\}}(1-p)+\one_{\{n\geq 2\}}p\sum_{i=1}^{n-1}u_iu_{n-i}.
    \]
    We deduce from this inductive relation that there exists a sequence $(\alpha_n)_{n\geq 1}$ independent from $p$ such that $\alpha_1=1$, $\alpha_n=\sum_{i=1}^{n-1}\alpha_i\alpha_{n-i}$ for all $n\geq 2$ and $u_n=\alpha_n(1-p)^n p^{n-1}$ for all $n\geq 1$. Since by definition (cf \eqref{def:unvhn}) $\sum_{n\geq 1}u_n=1$ whatever $p<1/2$, we deduce that for all $p<1/2$,
    \begin{equation}
        \label{eq:linksumalphap}
        \sum_{n\geq1}\alpha_n p^n(1-p)^n=p.
    \end{equation}
    Then, using classical results for derivative of function series, we have that
    \begin{equation}\label{eq:derivee}
    \sum_{n\geq 1} n(1-\beta)^{n-1}u_n=F'(1-\beta), \text{ with } F(y)=\sum_{n\geq 1}u_n y^n.
    \end{equation} 
    Moreover,  
    \begin{equation}\label{eq:F}
        F(y)= \frac{1}{p}\sum_{n\geq 1}\alpha_n [p(1-p)y]^n=\frac{1}{p}\sum_{n\geq 1}\alpha_n [\tilde{p}(y)(1-\tilde{p}(y))]^n= \frac{\tilde{p}(y)}{p},
    \end{equation}
    with $\tilde{p}(y)=\frac{1}{2}(1-\sqrt{1-4yp(1-p)})$, which implies that $\tilde{p}(y)$ is for all $y\in [0,1]$ the unique root in $[0,1/2]$ of the following polynomial $X(1-X)=y \, p \, (1-p)$; and where the last equality is a consequence of~\eqref{eq:linksumalphap} which is true for any $p<1/2$. In addition with Equation~\eqref{eq:derivee} and the definition of $x$ (cf Formula~\eqref{def:x}), we finally deduce that
    \begin{equation}\label{eq:sumbetaun}
    \sum_{n\geq 1}n(1-\beta)^{n-1}u_n=\frac{\tilde{p}'(1-\beta)}{p}=\frac{1-p}{x}.
    \end{equation}
    \smallskip
    
    We now deal with the sum $\sum_{n\geq 1}n(1-\beta)^{n-1}v_{g,n}$ for all $g\in \N^*$. Similarly, we first give an inductive relation on the sequence $(v_{g,n})_{1\leq g\leq n}$, by using again the Markov property at the time of the first event:
    \begin{equation*}
	v_{g,n}= (1-p)\one_{\{g=n=1\}} + 2 \, p \; \one_{2\leq g\leq n} \; \sum_{i=1}^{n-1}\frac{n-i}{n}v_{g-1,n-i}u_{i},
	\end{equation*}
	where the factor $(n-i)/n$ comes from the probability that the chosen leaf belongs to the sub-tree with $n-i$ leaves. From this last equation, we deduce by induction on the parameter $g$ that
	\begin{align}
	    v_{1,n}&=\one_{\{n=1\}}(1-p)\nonumber\\
	    \text{for all $g\geq2$, } v_{g,n}&=\one_{\{n\geq g\}}\frac{2^{g-1}}{n}(1-p)^n p^{n-1}\gamma_{g,n}\label{rec:v}\\
	    &\qquad\text{with }\gamma_{g,n}=\sum_{i=1}^{n-(g-1)}\alpha_i\gamma_{g-1,n-i}\one_{\{n\geq g\geq 3\}}\text{ and }\gamma_{2,n}=\alpha_{n-1}\one_{\{n\geq 2\}}.\nonumber
	\end{align}
	 Then, as $p<1/2$, for all $g\geq3$, 
	\begin{align*}
	    \sum_{n\geq 1}n(1-\beta)^{n-1}v_{g,n}&=\frac{2^{g-1}}{p(1-\beta)}\sum_{n\geq g}[(1-\beta)p(1-p)]^n\gamma_{g,n}\\
	    &=\frac{2^{g-1}}{p(1-\beta)}\sum_{n\geq g}\sum_{i=1}^{n-g+1}\alpha_i[(1-\beta)p(1-p)]^{i}\gamma_{g-1,n-i}[(1-\beta)p(1-p)]^{n-i},\\
	    &=\frac{2^{g-1}}{p(1-\beta)}\left(\sum_{n\geq 1} \alpha_n[(1-\beta)p(1-p)]^{n}\right)\left(\sum_{n\geq g-1}\gamma_{g-1,n}[(1-\beta)p(1-p)]^{n}\right),\nonumber
	    \end{align*}
	    where the last equality follows from the identification of a Cauchy product of series. Thus, using again an induction and then Equation~\eqref{eq:linksumalphap} and the definition of $\tilde{p}$, we obtain for all $g\geq3$, 
	    \begin{align}
	   \sum_{n\geq 1}n(1-\beta)^{n-1}v_{g,n}
	   &=\frac{2^{g-1}}{p(1-\beta)}\left(\sum_{n\geq 1} \alpha_n[(1-\beta)p(1-p)]^{n}\right)^{g-2}\left(\sum_{n\geq 2}\alpha_{n-1}[(1-\beta)p(1-p)]^{n}\right),\nonumber\\
	   &=\frac{2^{g-1}}{p(1-\beta)}\Big(\tilde{p}(1-\beta)\Big)^{g-1}(1-\beta)p(1-p) ,\nonumber\\
	   &=(1-x)^{g-1}(1-p),\label{eq:sum2}
	\end{align}
	which gives the value of the second sum. 
 Moreover, as $\sum_{n\geq 1}n(1-\beta)^{n-1}v_{1,n}=1-p$ and \\ $\sum_{n\geq 1}n(1-\beta)^{n-1}v_{2,n} = (1-p)(1-x)$, Formula \eqref{eq:sum2} is also true for $g \in \{1,2\}$.\\
    
    Finally, using Equations~\eqref{eq:sumbetaun} and~\eqref{eq:sum2}  in~\eqref{eq:HcondR1}, we obtain Equation~\eqref{eqPhr}, which concludes the proof of Lemma~\ref{lemma:GWgen}.

\end{proof}

\begin{proof}[Proof of Lemma~\ref{lem:probaanc}]
Using the notation and the arguments of the proofs of Proposition~\ref{th:ancestralresistant} and Lemma~\ref{lemma:GWgen}, we have
\begin{align*}
    \P(\A_1^N)&=\frac{\P(\RR^{\T_N}\cap \EE^{\T_N})}{\P(\EE^{\T_N})}=\frac{2p_N\P(\RR^{\T_N})\P(\RR_0^{\T_N})}{1-\gamma_N}\\
    &=\frac{2p_N}{(1-\gamma_N)}\sum_{n=1}^\infty u_n n\beta_N(1-\beta_N)^{n-1}\times \sum_{n=1}^\infty u_n (1-\beta_N)^{n}\\
    &=\frac{2p_N}{(1-\gamma_N)}\frac{\beta_N(1-p_N)}{x_N}\frac{1-x_N}{2p_N}\\
    &=\frac{1-x_N}{x_N(1-\gamma_N)}\gamma_N.
\end{align*}
where the third line is implied by~\eqref{eq:sumbetaun} and \eqref{eq:F} and the last one by $\beta_N(1-p_N)=\gamma_N$. Finally, to deduce~\eqref{eq:1anc}, it remains to find an approximation of $x_N$. 
Let notice that $1-p_N=\frac{d_0+\gamma_N b_0}{\delta_0}$ and $1-\beta_N=\frac{d_0(1-\gamma_N)}{d_0+\gamma_N b_0}$ then
\begin{align}
    x_N&=\left(1-4p_N(1-p_N)(1-\beta_N)\right)^{1/2}\nonumber\\
    &=\left(1-4\frac{(1-\gamma_N)^2b_0d_0}{\delta_0^2}\right)^{1/2}\nonumber\\
    &=\frac{\lambda_0}{\delta_0}\left(1+4\frac{b_0d_0 \gamma_N(2-\gamma_N)}{\lambda_0^2}\right)^{1/2}\nonumber\\
    &=\frac{\lambda_0}{\delta_0}+4\frac{b_0d_0}{\lambda_0 \delta_0}\gamma_N+\underset{N\to\infty}{o}(\gamma_N).\label{approx:xn}
\end{align}
Hence $(1-x_N)/(x_N(1-\gamma_N))$ converges to $2b_0/\lambda_0$ when $N$ increases to $\infty$, and we conclude \eqref{eq:1anc}.\\

Let us now deal with~\eqref{eq:k2anc}. To this aim, we denote by $\RR_k^{\T_N}$ the event that the tree $\T_N$ presents exactly $k$ \textit{ancestral resistant} cells. Notice that $\EE^{\T_N}\subset\RR^{\T_N}_k$ as soon as $k\geq 2$. Then
\begin{align*}
    \sum_{k\geq 2}k\P(\A_k^N)&=\sum_{k\geq 2}k\frac{\P\left(\RR^{\T_N}_k\cap \EE^{\T_N}\right)}{\P\left( \EE^{\T_N}\right)}=\frac{1}{1-\gamma_N}\sum_{k\geq 2}k\P\left(\RR^{\T_N}_k\right)\\
    &=\frac{1}{1-\gamma_N}\sum_{k\geq 2}k\sum_{n\geq k}\binom{n}{k}\beta_N^k(1-\beta_N)^{n-k}u_n\\
    &=\frac{1}{1-\gamma_N}\sum_{n\geq 2}u_n \sum_{k=2}^n k\binom{n}{k}\beta_N^k(1-\beta_N)^{n-k}\\
    &=\frac{1}{1-\gamma_N}\sum_{n\geq 2}u_n\left(n\beta_N-n\beta_N(1-\beta_N)^{n-1}\right)\\
    &=\frac{\beta_N}{1-\gamma_N}\left(\frac{1-p_N}{1-2p_N}-\frac{1-p_N}{x_N}\right),
\end{align*}
where the last equality is a consequence of \eqref{eq:sumbetaun}. In addition with \eqref{approx:xn} and the fact that $\beta_N(1-p_N)=\gamma_N$,
\begin{align*}
    \sum_{k\geq 2}k\P(\A_k^N)&=\frac{\gamma_N}{1-\gamma_N}\frac{1}{(1-2p_N)x_N}\left(x_N-1+2p_N\right)\\
    &=\frac{\gamma_N}{1-\gamma_N}
    \frac{1}{(1-2p_N)x_N}\left(\frac{\lambda_0}{\delta_0}+4\frac{b_0d_0}{\lambda_0\delta_0}\gamma_N+\underset{N\to\infty}{o}(\gamma_N)-\frac{\lambda_0}{\delta_0}-\frac{2b_0}{\delta_0}\gamma_N\right)\\
    &=\frac{2b_0\delta_0^2}{\lambda_0^3}\gamma_N^2+\underset{N\to\infty}{o}(\gamma_N^2),
\end{align*}
which concludes the proof of Lemma~\ref{lem:probaanc}.
\end{proof}

In the next section, we will see how to use Lemma~\ref{lem:probaanc} and the random variables $G_N$ and $T_N$ to compute the expected SFS of $Z^N$ with neutral mutations starting with $N$ sensitive cells.

\section{Expectation of the \textit{Site frequency spectrum}}
\label{sec:SFS}

We are now ready to prove Theorem \ref{mainth} and Theorem \ref{mainth2}, presented in Section~\ref{sec:model}, which give formulas and approximations when $N$ is large of the expected \textit{site frequency spectrum (SFS)} of a rescued population. As a reminder, the SFS is defined, for all $i\in \N^*$, as $S_i(t)$ the number of neutral mutations carried by exactly $i$ resistant cells alive at time $t$. To this aim, we will decompose $S_i^N(t)$ for all $i\in\N$ into two parts: $S_i^N(t)=\Sb_i^N(t)+\Su_i^N(t)$ where

\begin{itemize}
    \item $\Sb_i^N(t)$ counts mutations that appeared in a resistant cell;
    \item $\Su_i^N(t)$ counts mutations that appeared in a sensitive cell.
\end{itemize}

Using such decomposition we will prove our two main theorems by studying independently those two quantities, then combining them to deduce our main results. Indeed, Theorem~\ref{mainth} follows from Lemma~\ref{lem:Sbieq} and Lemma~\ref{lem:Sui} and Theorem~\ref{mainth2} follows from Lemma~\ref{lem:Sbxeq} and Lemma~\ref{lem:expSuiNapp}. 

Each of the following two subsections focuses on one of the two quantities above.

\subsection{Number of neutral mutations appeared in a resistant cell}
\label{ssec:sfsR}
The aim of this subsection is the study of $\E[\Sb_i^N(t_N)]$ for all $i\in \N^*$ and all $t>0$. To this aim, we will decompose the quantity into two parts such that, for all $i\in \N^*$, $t>0$,
  \begin{equation}
       \label{eq:expSbi}
       \E\left[\Sb_i^N(t_N)\right]=P(N,i) +R(N,i),
   \end{equation}
   where
   \begin{align}
       &P(N,i):=N^{1+\lambda_1t-\alpha}\frac{\delta_0(1-x_N)\gamma w }{(1-\gamma _N)} \int_0^{t_N} h_i(e^{\lambda_1\,(t_N-s)}) e^{-(\lambda_1 + x_N \delta_0)\, s} \;ds \label{eq:PNi}\\
       &\text{with }h_i(x):= \int_0^{(x-1)/(x-\frac{d_1}{b_1})} \frac{1-y}{1-\frac{d_1}{b_1}y}y^{i-1}dy,\label{def:hi}
       \end{align}
       and $R(N,i)$ will be given in \eqref{eq:RNidef}.\\
       $P(N,i)$ represents the part provided by the progeny that include exactly one \textit{ancestral resistant} cell
       and $R(N,i)$ represents the part provided by the progeny that carry two or more \textit{ancestral resistant} cells. When $N\to\infty$, we will be able to prove that this second part is negligible with respect to $P(N,i)$ in every situation of our interest.\\

Firstly, we state a result in the case where $i$ is fixed (see Lemma~\ref{lem:Sbieq}). Then we deal with the cases where $i$ is of order $N^{\lambda_1 t}$ (see Lemma~\ref{lem:Sbxeq}). $N^{\lambda_1 t}$ corresponds to the order size of the resistant cells number at time $t_N$. We also deduce a theoretical approximation when $N$ tends to infinity of $\E[\Sb_{x_1,x_2}^N(t)]$, where $\Sb^N_{x_1,x_2}(t)$ is the number of mutations (that appeared in a resistant cell) carried by a number of resistant cells between $x_1e^{\lambda_1 t_N}$ and  $x_2e^{\lambda_2 t_N}$  at time $t$, 
\begin{equation}\label{def_Sbarx1x2}
\Sb_{x_1,x_2}^N(t)=\sum_{i\in (x_1e^{\lambda_1 t_N},x_2e^{\lambda_1 t_N})} \Sb^N_i(t).
\end{equation}

\begin{Lem}
\label{lem:Sbieq}
For all $i\in \N^*$, $t>0$, 
   \begin{equation*}
        \E\left[\Sb_i^N(t_N)\right]\underset{N\to\infty}{\sim} I(i)\frac{2b_0\gamma\omega}{\lambda_1+\lambda_0}N^{1+\lambda_1 t-\alpha}, \text{ with } I(i):=\int_0^1\frac{1-y}{1-d_1y/b_1}y^{i-1}dy,
   \end{equation*}
   and
   \begin{equation}
       \underset{N\to\infty}{\limsup}\;\underset{i\in\N^*}{\sup}\frac{R(N,i)}{N^{1+\lambda_1t-2\alpha}}<\infty. 
       \label{eq:RNi}
   \end{equation}
\end{Lem}
\medskip
\begin{Lem}
\label{lem:Sbxeq}
 For all $t>0$, 
 \begin{itemize}
     \item[$(i)$] for $i_N=\lceil xN^{\lambda_1t}\rceil +1$ with $x> 0$, 
 \begin{equation*}
  \E\left[\Sb_{i_N}^N(t_N)\right]\underset{N\to\infty}{\sim} b_0\gamma \omega \lambda_1 K'(x) \, N^{1-\lambda_1 t -\alpha},
 \end{equation*}
 where $K'$ is the derivative of the function $K$ defined in \eqref{def:K} and 
 \begin{equation}
       \underset{N\to\infty}{\limsup}\;\underset{i\in\N^*}{\sup}\frac{R(N,i_N)}{N^{1-\lambda_1t-2\alpha}}<\infty.\label{eq:RNiN}
   \end{equation}
 \item[$(ii)$] Moreover for all $x_1,x_2\in (0,\infty]$ with $x_1<x_2$,
   \begin{equation*}
        \E\left[\Sb_{x_1,x_2}^N(t_N)\right]\underset{N\to\infty}{\sim} b_0 \gamma \omega \lambda_1 \big(K(x_1)-K(x_2))\big) \, N^{1 -\alpha}, 
       \end{equation*}
        where $K$ is defined in \eqref{def:K}.
 \end{itemize}
\end{Lem}
The rest of the subsection is devoted to the proofs of these results.

\begin{proof}[Proof of~\eqref{eq:expSbi}, \eqref{eq:PNi}]
    To obtain our result, we sum on all trees started by sensitive initial cells and structured the sum using the number of \textit{ancestral resistant} cells that appeared in these trees, i.e. for all $i\in \N$ and all $N\in \N$,  
    \begin{equation*}
        \E\left[\Sb_i^N(t_N)\right]=\E\left[\sum_{j=1}^N\sum_{k=1}^\infty \one_{\A_k^j}\sum_{l=1}^k S^{(1),j,l}_i(t_N-T^{j,l,k}_N) \right],
    \end{equation*}
    where for all $j\in \N$ and $k\in\N$, $\A_k^{j,N}$ is the event that there is exactly $k$ \textit{ancestral resistant} cells in the $j^{th}$ tree and $(T^{j,1,k}_N,...,T^{j,k,k}_N)$ is the random vector of appearance times of the $k$ \textit{ancestral resistant} cells of the $j^{th}$ tree conditioned on $\A_k^{j,N}$. For all $k\in\N$, the sequence $(T^{j,1,k}_N,...,T^{j,k,k}_N)_{j\in\N}$ is a i.i.d. sequence of vectors and, without loss of generality, we assume that the law of these vectors are exchangeable. Finally, for all $i\in \N$, $(S^{(1),j,l}_i)_{j,l\in\N}$ is a sequence of i.i.d random functions that give the SFS of the cell population issued from one resistant cell and this sequence is independent from all other random variables of the model. Indeed, $\Sb^N_i $ counts the mutations that appeared in resistant cells only. From this previous consideration and the fact that the $N$ trees issued from the $N$ sensitive initial cells are i.i.d., we deduce
    \begin{align*}
        \E\left[\Sb_i^N(t_N)\right]
        &=N\sum_{k=1}^\infty \P(\A_k^{1,N})k \E\left[S^{(1),1,1}_i(t_N-T^{1,1,k}_N) |\A_k^{1,N} \right].
    \end{align*}
    In what follows, for simplicity, $T^{1,1,k}_N$ will be denoted by $T^k_N$, $\A_k^{1,N}$ by $\A_k^N$ and $S^{(1),1,1}_i$ by $S^{(1)}_i$ for all $k,i\in\N$. Thus, by denoting $f_X$ the distribution function of a r.v. $X$ that admits such function, we have
     \begin{align}
        \E\left[\Sb_i^N(t_N)\right]= &N\P(\A_1^N)\int_0^{t_N}E\left[ S^{(1)}_i(t_N-s) \right]f_{T^{1}_N}(s)ds +R(N,i)  \label{eq:Sbarint} \\
        &\text{with } R(N,i)=N\sum_{k=2}^\infty \P(\A_k^N)k \int_0^{t_N}\E\left[S^{(1)}_i(t_N-s) \right]f_{T^k_N}(s)ds. \label{eq:RNidef}
    \end{align}
    In particular, this corresponds to the splitting found in \eqref{eq:expSbi}.\\
    
    Then from Lemma~\ref{lemS1}, we deduce that, for all $t\geq 0$,
    \begin{equation}\label{eqres:SFS1}
    \mathbb{E}\left[S^{(1)}_i(t)\right]= \omega  e^{\lambda_1 t} \int_0^{ \frac{e^{\lambda_1 t}-1}{e^{\lambda_1 t}-d_1/b_1}} \frac{ 1-y}{ 1-\frac{d_1}{b_1}y}y^{i-1}dy=\omega  e^{\lambda_1 t} h_i(e^{\lambda_1 t}).
    \end{equation}
     Indeed, $S^{(1)}_i(t)$ corresponds to the SFS at time $t$ of a birth and death branching process with neutral mutations accumulated at each division, starting from a single individual and with birth rate $b_1$ and death rate $d_1$. Notice that it  does not depend on $N$. \\
     
     Moreover, recall that the law of $T^{1}_N$ is given by \eqref{eqfT} of Proposition~\ref{th:ancestralresistant}. Then we use \eqref{eqres:SFS1}, \eqref{eq:1anc} and \eqref{eqfT} into \eqref{eq:Sbarint} to deduce \eqref{eq:PNi}. 
     \end{proof}
     
     \begin{proof}[Proof of Lemma~\ref{lem:Sbieq}]
     Let us prove \eqref{eq:RNi}. To this aim, we use \eqref{eq:RNidef} and \eqref{eqres:SFS1} to deduce that for all $N,i\in\N$
     \begin{align*}
         R(N,i)&\leq  N\sum_{k=2}^\infty k\P(\A_k^N) \int_0^{\infty}\omega e^{\lambda_1(t_N-s)}\left(\int_0^1\frac{(1-y)y^{i-1}}{1-d_1y/b_1}dy\right)f_{T^k_N}(s)ds\\
         &\leq \omega \frac{b_1}{\lambda_1}N^{1+\lambda_1t}\sum_{k=2}^\infty k \P(\A_k^N),
     \end{align*} where the last inequality is due to the upper bounds $\int_0^1\frac{(1-y)y^{i-1}}{1-\frac{d_1}{b_1}y}dy \leq \frac{b_1}{\lambda_1} $ and $\int_0^{\infty}e^{-\lambda_1s}f_{T^k_N}(s)ds \leq 1$.
    The result~\eqref{eq:RNi} is then a consequence of Equation~\eqref{eq:k2anc} of Lemma~\ref{lem:probaanc}.\\

	We can now prove Lemma~\ref{lem:Sbieq}. In view of \eqref{eq:expSbi} and \eqref{eq:RNi},  it is sufficient to show that \begin{equation}\label{conv:delta}
	\frac{P(N,i)}{N^{1+\lambda_1t-\alpha}}\underset{N\to\infty}{\rightarrow}2b_0\omega \gamma \frac{I(i)}{\lambda_0+\lambda_1}.
	\end{equation}
  We know from \eqref{eq:PNi} that,
	$$\frac{P(N,i)}{N^{1+\lambda_1t-\alpha}}= \frac{\omega \gamma\delta_0(1-x_N) }{1-\gamma _N} \int_0^{\infty} g_{N}(s) \;ds, \text{ with } g_{N}(s):=\one_{s\leq t_N} h_i(e^{\lambda_1\,(t_N-s)}) e^{-(\lambda_1 + x_N \delta_0)\, s}.
	$$ 
	Then, for all $s\geq 0$, we have the following convergence
	\begin{equation*}
	    g_{N}(s)\underset{N\to\infty}{\rightarrow} I(i)e^{-(\lambda_0+\lambda_1)s}.
	\end{equation*} 
	Moreover for all $s\geq 0$ and $N\in \N$,
		$ |g_{N}(s)|\leq  I(i)e^{-\lambda_1 s}$,
	which defined an integrable function on $(0,\infty)$. We thus conclude with the Dominated convergence theorem  that $\int_0^{\infty} g_{N}(s) \;ds$ tends to $ I(i)/(\lambda_0+\lambda_1)$ when $N$ tends to $\infty$. Finally, according to \eqref{approx:xn}, we have that $\delta_0(1-x_N)/(1-\gamma_N)$ tends to $ 2b_0$ when $N$ tends to infinity. Combining the last two limits gives \eqref{conv:delta}, which ends the proof  of Lemma~\ref{lem:Sbieq}.
	\end{proof}

\begin{proof}[Proof of Lemma~\ref{lem:Sbxeq}]
Let us first prove~\eqref{eq:RNiN}. From~\eqref{eq:RNidef} and \eqref{eqres:SFS1}, we know that, with the notation of the previous proof, 
\begin{equation*}
R(N,i_N)=N\sum_{k=2}^\infty \P(\A_k^N)k \int_0^{t_N}\omega e^{\lambda_1(t_N-s)}h_{i_N}(e^{\lambda_1(t_N-s)})f_{T^k_N}(s)ds,
\end{equation*}
with $h_i(x)=\int_0^{\frac{x-1}{x-d_1/b_1}}\frac{1-y}{1-d_1y/b_1}y^{i-1}dy$. Since $\lambda_1/b_1\leq (1-d_1y/b_1)$ for all $y\in (0,1)$,
\begin{equation}\label{maj:hi}
	h_i(x)\leq \frac{b_1}{\lambda_1} \tilde h_i(x),
	\end{equation}
	with $\tilde h_i(x):=	\frac{1}{i}\left(\frac{x-1}{x-\frac{d_1}{b_1}}\right)^i-\frac{1}{i+1}\left(\frac{x-1}{x-\frac{d_1}{b_1}}\right)^{i+1} =\frac{1}{i(i+1)}\left(1-\frac{\frac{\lambda_1}{b_1}}{x-\frac{d_1}{b_1}}\right)^i\left[1+i\frac{\frac{\lambda_1}{b_1}}{x-\frac{d_1}{b_1}}\right].$
Thus 
\begin{equation*}
  \frac{R(N,i_N)}{N^{1-\lambda_1 t -2\alpha}}\leq \frac{b_1\omega  N^{2\lambda_1 t+2\alpha} }{\lambda_1 i_N (i_N+1)} \sum_{k=2}^\infty \P(\A_k^N)k \int_0^{\infty} \tilde{g}_N(s) f_{T^k_N}(s)ds,
\end{equation*}
with
\begin{align}
\tilde{g}_N(s)&:=\one_{\{s\in[0,t_N]\}}e^{-\lambda_1 s} \left(1-\frac{\frac{\lambda_1}{b_1}}{e^{\lambda_1(t_N-s)}-\frac{d_1}{b_1}}\right)^{i_N}\left[1+\frac{i_N\frac{\lambda_1}{b_1}}{e^{\lambda_1(t_N-s)}-\frac{d_1}{b_1}}\right] \label{eq:gN}\\
&\leq e^{-\lambda_1 s}\left[1+i_N\frac{\lambda_1}{b_1}\frac{e^{-\lambda_1t_N}e^{\lambda_1 s}}{1-\frac{d_1}{b_1}e^{-\lambda_1(t_N-s)}}\right] \one_{\{s\in[0,t_N]\}} \nonumber\\
&\leq 1+(x+2)\leq x+3,\nonumber
\end{align}
for $N$ sufficiently large, since $i_N=\lceil xe^{\lambda_1 t_N}\rceil +1$. Hence,
\begin{equation}\label{maj:RNiN}
  \frac{R(N,i_N)}{N^{1-\lambda_1 t -2\alpha}}\leq \frac{b_1\omega  N^{2\alpha} }{x^2 \lambda_1} \left(x+3\right) \sum_{k=2}^\infty \P(\A_k^N)k.
\end{equation}
We finally find~\eqref{eq:RNiN} by using the approximation given by \eqref{eq:k2anc} of Lemma~\ref{lem:probaanc}.\\

We then prove Lemma~\ref{lem:Sbxeq} $(i)$. We fix $x>0$ and recall that $i_N=\lceil xN^{\lambda_1 t}\rceil +1$. Recall the definition $P(N,i)$ given by~\eqref{eq:PNi}. In view of \eqref{eq:RNiN}, it is sufficient to prove that 
	\begin{equation}\label{eq:limtoshow}
	    \frac{P(N,i_N)}{N^{1-\lambda_1 t-\alpha}}\underset{N\to\infty}{\longrightarrow} b_0\gamma \omega \lambda_1 K'(x).
	\end{equation}
	According to the change of variable used in the proof of Lemma~\ref{lemS1}, we know that, for $s\in[0,t_N]$, 
 \begin{equation*}
     h_i(e^{\lambda_1 (t_N-s)}) = \frac{\lambda_1^2}{b_1} \int_s^{t_N} \frac{1}{(e^{\lambda_1 (t_N-u)}-\frac{d_1}{b_1})^2} \Big( \frac{e^{\lambda_1 (t_N-u)}-1}{e^{\lambda_1 (t_N-u)}-\frac{d_1}{b_1}}\Big)^{i-1} du.
 \end{equation*}
 Then, noticing that $$\int_0^u e^{-(\lambda_1+x_N\delta_0)s}ds = \frac{(1-e^{-(\lambda_1+x_N\delta_0)u})}{\lambda_1+x_N \delta_0},$$
 we deduce from \eqref{eq:PNi} by an integral exchange that,
 \begin{equation}\label{PNiN}
     P(N,i_N)=N^{1+\lambda_1t-\alpha} \frac{ \lambda_1^2 \delta_0(1-x_N)\gamma\omega }{b_1 (1-\gamma _N)(\lambda_1+x_N\delta_0)} \int_0^{+\infty} l_N(u) e^{-2\lambda_1 \, t_N} \;du,
 \end{equation}
 where
 $$l_N(u):=\one_{\{u\in[0,t_N]\}}\frac{(1-e^{-(\lambda_1+x_N\delta_0) u})}{(e^{-\lambda_1u}-\frac{d_1}{b_1}e^{-\lambda_1 t_N})^2} \Big( \frac{1-e^{-\lambda_1 (t_N-u)}}{1-\frac{d_1}{b_1}e^{-\lambda_1 (t_N-u)}}\Big)^{i_N-1}. $$
However, for $u\leq t_N$, we have,
\begin{align*}
    \Big( \frac{1-e^{-\lambda_1 (t_N-u)}}{1-\frac{d_1}{b_1}e^{-\lambda_1 (t_N-u)}}\Big)^{i_N-1}&=\exp\Big((i_N-1)\log\left(1 - \frac{\lambda_1}{b_1} \frac{e^{-\lambda_1(t_N-u)}}{1-\frac{d_1}{b_1}e^{-\lambda_1(t_N-u)}}\right)\Big)\\
    &\leq \exp\Big(-(i_N-1)\frac{\lambda_1}{b_1} e^{-\lambda_1(t_N-u)}\Big)
\end{align*} 
with $i_N=\lceil xe^{\lambda_1t_N}\rceil +1$, $x> 0$. Hence for all $u>0$, $$l_N(u)\leq \frac{(e^{\lambda_1 u})^2}{(1-\frac{d_1}{b_1}e^{\lambda_1 u})^2} \exp(-x\frac{\lambda_1}{b_1}e^{\lambda_1 u}).$$
Such upper bound defines an integrable function on $\R_+$.
Moreover, notice that, when $N$ tends to $+\infty$, $\delta_0(1-x_N)/(1-\gamma_N)$ tends to $ 2b_0$ and for all $u>0$, \begin{equation}\label{conv:lN}
    l_N(u) \underset{N\to\infty}{\to} (1-e^{-(\lambda_1+\lambda_0) u})e^{2\lambda_1u} \exp\left(-x\frac{\lambda_1}{b_1}e^{\lambda_1u}\right).
\end{equation}
Then we can deduce from Formula~\eqref{PNiN}, the convergence result \eqref{eq:limtoshow} by the dominated convergence theorem. That concludes the proof of Lemma~\ref{lem:Sbxeq} $(i)$.

	\medskip We end the proof by proving the second part of the Lemma~\ref{lem:Sbxeq} $(ii)$. According to~\eqref{eq:expSbi} and \eqref{def_Sbarx1x2},
	\begin{equation*}
	\E\left[\Sb_{x_1,x_2}^N(t)\right]=\sum_{i\in (x_1N^{\lambda_1 t},x_2N^{\lambda_1 t})} P(N,i)+\sum_{i\in (x_1N^{\lambda_1 t},x_2N^{\lambda_1 t})} R(N,i).
	\end{equation*}
	Then from \eqref{maj:RNiN}, we obtain the following upper bound
	\begin{align*}
	    \frac{1}{N^{1-2\alpha}}\sum_{i\in (x_1N^{\lambda_1 t},x_2N^{\lambda_1 t})} R(N,i)&\leq N^{-\lambda_1 t}\sum_{i\in (x_1N^{\lambda_1 t},x_2N^{\lambda_1 t})} \frac{b_1\omega (ie^{-\lambda_1 t_N}+2) N^{2\alpha} }{(ie^{-\lambda_1 t_N}+1)^2 \lambda_1}  \sum_{k=2}^\infty \P(\A_k^N)k\\
	    &\leq (x_2-x_1) \frac{b_1\omega (x_2+2)  }{(x_1+1)^2 \lambda_1} N^{2\alpha} \sum_{k=2}^\infty \P(\A_k^N)k,
	\end{align*}
	and from \eqref{eq:k2anc}, we deduce that 
	$$
	\limsup_{N\in\N} \frac{1}{N^{1-2\alpha}}\sum_{i\in (x_1N^{\lambda_1 t},x_2N^{\lambda_1 t})} R(N,i) <\infty.
	$$
	Then, noticing that $\frac{\delta_0(1-x_N)}{(1-\gamma_N)(\lambda_1+x_N\delta_0)}$ tends to $\frac{2b_0}{\lambda_1+\lambda_0}$ when $N$ tends to infinity (see Formula~\eqref{approx:xn}), we deduce from Formula~\eqref{PNiN} that Lemma~\ref{lem:Sbxeq} $(ii)$ will be proved as soon as we prove 
	\begin{equation*}\label{eq:limtoshow3}
	    \tilde{\Omega}_N:=\frac{\lambda_1}{b_1} N^{\lambda_1 t} \sum_{i_N\in (x_1N^{\lambda_1 t},x_2N^{\lambda_1 t})}\int_0^{+\infty} l_N(u) e^{-2\lambda_1 \, t_N} \;du, \underset{N\to\infty}{\longrightarrow} K(x_1)-K(x_2),
	\end{equation*}
 with,
 $$l_N(u)=\one_{\{u\in[0,t_N]\}}\frac{(1-e^{-(\lambda_1+x_N\delta_0) u})}{(e^{-\lambda_1u}-\frac{d_1}{b_1}e^{-\lambda_1 t_N})^2} \Big( \frac{1-e^{-\lambda_1 (t_N-u)}}{1-\frac{d_1}{b_1}e^{-\lambda_1 (t_N-u)}}\Big)^{i_N-1}. $$

	However, by calculus, we obtain
	\begin{align*}
	   \tilde{\Omega}_N&=\int_0^{t_N}\frac{1-e^{(-\lambda_1+\delta_0 x_N)u}}{e^{-\lambda_1 u}-\frac{d_1}{b_1}e^{-\lambda_1 t_N}} \bigg[\left(1-\frac{\lambda_1}{b_1}\frac{e^{-\lambda_1(t_N-u)}}{1-\frac{d_1}{b_1}e^{-\lambda_1(t_N-u)}}\right)^{\lfloor x_1N^{\lambda_1 t}\rfloor}\\
	   &\hspace{6cm}-\left(1-\frac{\lambda_1}{b_1}\frac{e^{-\lambda_1(t_N-u)}}{1-\frac{d_1}{b_1}e^{-\lambda_1(t_N-u)}}\right)^{\lfloor x_2N^{\lambda_1 t}\rfloor+1}\bigg]ds.
	\end{align*}
We conclude, arguing as previously, with the dominated convergence theorem.
\end{proof}

\subsection{Number of neutral mutations appeared in a sensitive cell}
\label{ssec:sfsS}

In this subsection, we deal with the second term that appears in the SFS, that is, $\Su_i^N(t_N)$, which counts the number of neutral mutations that appeared in a sensitive cell and that can be found in exactly $i$ resistant cells at time $t_N$. As previously, we will divide the quantity into two parts such that, for all $i\in \N^*$, $t>0$,
  \begin{equation}
       \label{eq:expSu}
       \E\left[\Su_i^N(t_N)\right]=Q(N,i) +\tilde{R}(N,i),
   \end{equation}
   where $Q(N,i)$ represents the weight provided by the trees that carry exactly one \textit{ancestral resistant} cell and can be written as
   \begin{align}
       &Q(N,i):=\frac{N\gamma_N(1-x_N)\delta_0\omega}{2(1-\gamma_N)}\int_0^{t_N}\kappa_i\left(e^{-\lambda_1 (t_N-s)}\right)(1+s\delta_0(1-x_N))e^{-s\delta_0x_N} \, ds \label{eq:QNi}\\
       &\text{with }\kappa_i(x):=\frac{\lambda_1^2}{b_1^2}\frac{x(1-x)^{i-1}}{(1-\frac{d_1}{b_1}x)^{i+1}};\label{def:hibis}
       \end{align}
       and where $\tilde{R}(N,i)$ represents the part provided by the trees that carry two or more \textit{ancestral resistant} cells. When $N\to\infty$, we will be able to prove once again that this part is negligible with respect to $Q(N,i_N)$.\\

Two results are deduced below using this decomposition. The first one deal with the case when $i$ is fixed. In this case, we prove that this part of the SFS is negligible compared with $\E[\Sb_i^N(t_N)]$, described in the previous section, and which is of order $N^{1+\lambda_1t-\alpha}$.
\begin{Lem}
\label{lem:Sui}
   For all $i\in \N^*$, $t>0$, when $N$ is large,
   \begin{equation}
       \label{eq:expSuiapp}
       \E\left[\Su_i^N(t_N)\right]=\underset{N\to\infty}{O}\left(N^{1-\alpha}\right).
   \end{equation}
\end{Lem}

When $i=i_N$ depends on the size of the population $N$, the result is more intricate. Contrary to the case studied in the previous section related to $\Sb_i$, we are not able to deal with the equivalent of $\E[\Su_{i_N}(t_N)]$, as we are not able to bound precisely $\tilde R(N,i_N)$. The main difficulty comes from the fact that the relationships between the \textit{ancestral resistant} cells have to be managed to deal with this quantity, which is currently beyond our reach. We however are able to derive a result for  $\E[\Su_{x_1,x_2}^N(t)]$, where $\Su^N_{x_1,x_2}(t)$ is the number of mutations (that appeared in a sensitive cell) carried by a number of resistant cells between $x_1e^{\lambda_1 t_N}$ and  $x_2e^{\lambda_2 t_N}$  at time $t$, 
\begin{equation}\label{def_Sux1x2}
\Su_{x_1,x_2}^N(t)=\sum_{i\in (x_1e^{\lambda_1 t_N},x_2e^{\lambda_1 t_N})} \Su^N_i(t).
\end{equation} 
Our result in this case is an exact asymptotic equivalent, contrary to the previous section.
Let us state it in the next lemma.
\begin{Lem}\label{lem:expSuiNapp}
 For all $t>0$, and for $x_1,x_2\in(0,\infty]$ with $x_1<x_2$,
   \begin{equation}
       \label{eq:expSuxapp}
      \E\left[\Su_{x_1,x_2}\big(t_N\big)\right] \underset{N\to\infty}{\sim} b_0\gamma\omega \lambda_1  \Big(L(x_1)-L(x_2)\Big) N^{1-\alpha} , 
       \end{equation}
       where $L$ is defined in \eqref{def:L}.
   
\end{Lem}

The end of the section is devoted to the proof of the previous lemmas.

\begin{proof}[Proof of~\eqref{eq:expSu} and~\eqref{eq:QNi}]
As indicated in the introduction of this section, we divide the quantity into two parts, the first one incorporates progeny that include exactly one \textit{ancestral resistant} cells, the second those that include two or more \textit{ancestral resistant} cells: for all $i\in\N$, $t>0$,
\begin{equation*}
   \E\left[\Su_i^N(t_N)\right]=N \sum_{k=1}^\infty \P(\A_k^N)\E\left[\Su_i^{(k)}(t_N)|\A_k^N\right],
\end{equation*}
where $\Su_i^{(k)}$ is the number of neutral mutations that appeared in a sensitive cell whose first ancestor has exactly $k$ \textit{ancestral resistant} cells in its progeny. Obviously, to obtain~\eqref{eq:expSu} we set 
\begin{equation}\label{def:QNRtN}
    Q(N,i):=N  \P(\A_1^N)\E\left[\Su_i^{(1)}(t_N)|\A_1^N\right] \quad \text{and} \quad \tilde{R}(N,i):=N \sum_{k=2}^\infty \P(\A_k^N)\E\left[\Su_i^{(k)}(t_N)|\A_k^N\right].
\end{equation}
Let us now find a precise expression for $Q$. We denote by $(N_\omega^j)_{j\in\N}$ a i.i.d. sequence of r.v. with the same law as $N_\omega$, and by $\tilde{Z}$ a branching birth and death process with birth rate $b_1$ and death rate $d_1$ starting from $1$ individual. Then using~\eqref{eq:1anc} and the notation of Proposition~\ref{th:ancestralresistant}, we find
\begin{align}
     Q(N,i)&=  \frac{N\gamma_N(1-x_N)}{x_N(1-\gamma_N)}\E\left[\sum_{j=1}^{G_N}N_\omega^{j}\one_{\{\tilde{Z}(t_N-T_N)=i\}}\right],\nonumber\\
     &=\frac{N\gamma_N(1-x_N)\omega}{2x_N(1-\gamma_N)}\E\left[G_N\E[\one_{\{\tilde{Z}(t_N-T_N)=i\}}|G_N]\right].\label{eq:equaQNi}
\end{align}
The law of $T_N|\{G_N=g\}$ is a Gamma law with parameters $(\delta_0,g)$, and from Formula (7.3) Chap 5 of \cite{Harris1963}, we know that for all $u\geq 0$, $i\geq 1$,
\begin{equation}\label{eq:probaZi}
\P\left(\tilde{Z}(u)=i\right)=\frac{\lambda_1^2e^{-\lambda_1 u}(1-e^{-\lambda_1 u})^{i-1}}{b_1^2(1-\frac{d_1}{b_1}e^{-\lambda_1 u})^{i+1}}=\kappa_i(e^{-\lambda_1 u}).
\end{equation}
In addition with~\eqref{eq:equaQNi}, we find
\begin{align*}
    Q(N,i)&=\frac{N\gamma_N(1-x_N)\omega}{2x_N(1-\gamma_N)}\sum_{g=1}^\infty g\P(G_N=g)\int_0^{t_N}\kappa_i\left(e^{-\lambda_1 (t_N-s)}\right)\frac{s^{g-1}e^{-\delta_0 s}}{(g-1)!}\delta_0^g \, ds\\
    &=\frac{N\gamma_N(1-x_N)\omega}{2(1-\gamma_N)}\int_0^{t_N}\kappa_i\left(e^{-\lambda_1 (t_N-s)}\right)e^{-\delta_0 s}\sum_{g=1}^\infty g(1-x_N)^{g-1}\frac{s^{g-1}\delta_0^g}{(g-1)!} \, ds.
\end{align*}
By noticing that for all $z\in\R$
$$
\sum_{g=1}^\infty g\frac{z^{g-1}}{(g-1)!}=\sum_{g=2}^\infty \frac{z^{g-1}}{(g-2)!}+\sum_{g=1}^\infty \frac{z^{g-1}}{(g-1)!}=(z+1)e^{z},
$$
we obtain equation~\eqref{eq:QNi}.
\end{proof}

\begin{proof}[Proof of Lemma~\ref{lem:Sui}]
Let us first deal with the term $\tilde{R}(N,i)$. Recalling~\eqref{def:QNRtN}, we have that
\begin{align}
   \tilde{R}(N,i)   &\leq N \sum_{k=2}^\infty \P(\A_k^N)k\E[G^{(k),1}]\E[N_w],\label{eq:majSuiN}
\end{align}
where $(G^{(k),1},...,G^{(k),k})$ is the vector of the generations of the $k$ \textit{ancestral resistant} cells of a progeny, knowing that this progeny contains exactly $k$ \textit{ancestral resistant} cells. We assume that the law of this vector is exchangeable without loss of generality. The r.h.s of~\eqref{eq:majSuiN} is obtained by considering that all neutral mutations appeared in the ancestors of an \textit{ancestral resistant} cell (in each progeny containing exactly $k$ \textit{ancestral resistant} cells) count in $\Su_i^{(k)}(t)$. Then using the notation of Section~\ref{sec:AR} and arguing as in \eqref{eq:HcondR1} and Lemma~\ref{lem:probaanc}, we have for all $k\geq 2$,
\begin{align*}
    \P(\A_k^N)\E[G^{(k),1}]&=\P(\RR_k^{\T_N}\cap\EE^{\T_N}|\EE^{\T_N})\sum_{g\geq 1}g\P\left(G^{\T_N}=g+1|\RR_k^{\T_N},\EE^{\T_N}\right)\\
    &=\frac{1}{1-\gamma_N}\sum_{g\geq 1}g\P\left(\{G^{\T_N}=g+1\} \cap \RR_k^{\T_N}\right)\\
    &=\frac{1}{1-\gamma_N}\sum_{g\geq 1}g\sum_{n\geq k}^{\infty}\binom{n}{k}\beta_N^k(1-\beta_N)^{n-k}v_{g+1,n},
\end{align*}
In addition with~\eqref{eq:majSuiN}, we find
\begin{align*}
   \tilde{R}(N,i)&\leq\frac{N\omega}{2(1-\gamma_N)}    \sum_{g\geq 1}g\sum_{n\geq 2}^{\infty} \sum_{k=2}^n k\binom{n}{k}\beta_N^k(1-\beta_N)^{n-k}v_{g+1,n},\\
    &\leq \frac{N\omega}{2(1-\gamma_N)}    \sum_{g\geq 1}g\beta_N\sum_{n\geq 2}^{\infty} (n v_{g+1,n} -n(1-\beta_N)^{n-1}v_{g+1,n}),\\
     &\leq \frac{N\omega\beta_N(1-p_N)}{2(1-\gamma_N)}    \sum_{g\geq 1}g((2p_N)^{g}-(1-x_N)^g).
\end{align*}
where the last inequality is a consequence of~\eqref{eq:sum2}. Then, noticing by \eqref{def:pNbetaN} that $\gamma_N=\beta_N(1-p_N)$,
\begin{align}
   \tilde{R}(N,i)&\leq \frac{N\omega\gamma_N}{2(1-\gamma_N)}   \left(\frac{2p_N}{1-2p_N} -\frac{1-x_N}{x_N} \right)\nonumber\\
  & \leq \frac{N\omega\gamma_N}{2(1-\gamma_N)(1-2p_N)x_N}   \left(x_N+2p_N-1\right)\nonumber\\
  &\leq \frac{N\omega\gamma_N^2}{2(1-\gamma_N)(1-2p_N)x_N} \frac{2b_0}{\delta_0}   \left(\frac{2d_0}{\lambda_0}-1\right) +\underset{N\to\infty}{o}(N^{1-2\alpha}),\label{ineq:tildeRNi}
\end{align}
according to \eqref{approx:xn}.

Let us now deal with the first term $Q(N,i)$. According to \eqref{eq:probaZi}, $\kappa_i(e^{-\lambda_1u})$ corresponds to a probability, then
\begin{align}
    Q(N,i)&\leq\frac{N\gamma_N\delta_0\omega}{2(1-\gamma_N)}\int_0^{\infty}(1+s\delta_0)e^{-s\delta_0x_N} \, ds\nonumber\\
    &\leq  \frac{N\gamma_N\delta_0\omega}{2(1-\gamma_N)}\left[\frac{1}{\delta_0x_N}+\frac{1}{\delta_0x_N^2}\right]=\underset{N\to\infty}{O}(N^{1-\alpha}).\label{ineq:QNi}
\end{align}
Equations~\eqref{ineq:tildeRNi} and \eqref{ineq:QNi} are sufficient to obtain \eqref{eq:expSuiapp} and to conclude Lemma~\ref{lem:Sui}.
\end{proof}

\begin{proof}[Proof of Lemma~\ref{lem:expSuiNapp}]
Let set $0<x_1<x_2$, then 
\begin{equation}\label{eq:decSu}
\E\left[\Su_{x_1,x_2}^N(t)\right]=\sum_{i\in (x_1N^{\lambda_1 t},x_2N^{\lambda_1 t})} Q(N,i)+\sum_{i\in (x_1N^{\lambda_1 t},x_2N^{\lambda_1 t})}\tilde{R}(N,i).
\end{equation}
Let us first deal with the last term of the r.h.s.. Note that each mutation is counted in only one of the elements of the sequence $(\Su_i^N)_{i\geq 1}$, corresponding to the exact number $i$ of resistant offspring from the cell in which this mutation occurred. Thus, the bound~\eqref{eq:majSuiN} is still valid for $\sum_{i\in (x_1N^{\lambda_1 t},x_2N^{\lambda_1 t})}\tilde{R}(N,i)$ as we bounded by adding all neutral mutations that appeared in the ancestors of all \textit{ancestral resistant} cells (considering progeny containing
at least $2$ \textit{ancestral resistant} cells). Thus, using arguments similar to those used to obtain~\eqref{ineq:tildeRNi}, we conclude that
\begin{equation}\label{ineq:sumtRNi}
   \sum_{i\in (x_1N^{\lambda_1 t},x_2N^{\lambda_1 t})}\tilde{R}(N,i)\leq \underset{N\to\infty}{O}(N^{1-2\alpha}).
\end{equation}

We finally deal with the first term of the r.h.s of \eqref{eq:decSu}. Using the definition of $Q(N,i)$ in \eqref{eq:QNi} and the fact that, for all $z< 1$, $$\sum_{i=k_1+1}^{k_2}z^{i}=\frac{z^{k_1}-z^{k_2}}{1-z},
$$ 
and that 
$$
\left(1-\frac{1-e^{-\lambda_1(t_N-s)}}{1-\frac{d_1}{b_1}e^{-\lambda_1(t_N-s)}}\right)^{-1}=\frac{1-\frac{d_1}{b_1}e^{-\lambda_1(t_N-s)}}{\frac{\lambda_1}{b_1}e^{-\lambda_1(t_N-s)}}
$$
we find
\begin{align}
   &\frac{1}{N^{1-\alpha}}\sum_{i= \lfloor x_1N^{\lambda_1 t}\rfloor +1}^{\lfloor x_2N^{\lambda_1 t}\rfloor} Q(N,i)\nonumber\\
   &= \frac{\gamma(1-x_N)\delta_0\omega\lambda_1^2}{2(1-\gamma_N)b_1^2}\int_0^{t_N}\frac{(1+s\delta_0(1-x_N))e^{-s\delta_0x_N-\lambda_1(t_N-s)}}{(1-\frac{d_1}{b_1}e^{-\lambda_1(t_N-s)})^2}\sum_{i= \lfloor x_1N^{\lambda_1 t}\rfloor +1}^{\lfloor x_2N^{\lambda_1 t}\rfloor}
   \left(\frac{1-e^{-\lambda_1(t_N-s)}}{1-\frac{d_1}{b_1}e^{-\lambda_1(t_N-s)}}\right)^{i-1}
   \, ds \nonumber\\
   &= \frac{\gamma(1-x_N)\delta_0\omega\lambda_1}{2(1-\gamma_N)b_1}\int_0^{t_N}\frac{(1+s\delta_0(1-x_N))e^{-s\delta_0x_N}}{1-\frac{d_1}{b_1}e^{-\lambda_1(t_N-s)}}\Bigg[\left(\frac{1-e^{-\lambda_1(t_N-s)}}{1-\frac{d_1}{b_1}e^{-\lambda_1(t_N-s)}}\right)^{\lfloor x_1N^{\lambda_1 t}\rfloor} \nonumber\\
   &\hspace{8.5cm}-\left(\frac{1-e^{-\lambda_1(t_N-s)}}{1-\frac{d_1}{b_1}e^{-\lambda_1(t_N-s)}}\right)^{\lfloor x_2N^{\lambda_1 t}\rfloor}\Bigg] \, ds. \label{eq:intersumQNi}
\end{align}
We now prove that the r.h.s converges, when $N$ increases, to $+\infty$. According to~\eqref{approx:xn}, 
\begin{equation}\label{conv:constant}
    \frac{\gamma(1-x_N)\delta_0\omega\lambda_1}{2(1-\gamma_N)b_1} \underset{N\to\infty}{\longrightarrow}  \frac{\gamma b_0\omega\lambda_1}{b_1}.
\end{equation}
Let define
\begin{align*}
\hat{g}_N(s):=\one_{\{s\in(0,t_N)\}}\frac{(1+s\delta_0(1-x_N))e^{-s\delta_0x_N}}{1-\frac{d_1}{b_1}e^{-\lambda_1(t_N-s)}}\Bigg[&\left(\frac{1-e^{-\lambda_1(t_N-s)}}{1-\frac{d_1}{b_1}e^{-\lambda_1(t_N-s)}}\right)^{\lfloor x_1N^{\lambda_1 t}\rfloor} \\
&-\left(\frac{1-e^{-\lambda_1(t_N-s)}}{1-\frac{d_1}{b_1}e^{-\lambda_1(t_N-s)}}\right)^{\lfloor x_2N^{\lambda_1 t}\rfloor}\Bigg].
\end{align*}
Using \eqref{approx:xn}, we prove that, for $N$ large enough,
\[
\text{ for all } s\in\R^+, \quad \hat{g}_N(s) \leq \frac{(1+s\delta_0)e^{-s\lambda_0}}{1-d_1/b_1},
\]
whose r.h.s. can be integrated on $(0,\infty)$ w.r.t $s$, and arguing as in \eqref{eq:gN}, we have that for all $s\in\R^+$
\[
\hat{g}_N(s)\underset{N\to\infty}{\longrightarrow} (1+2b_0s)e^{-\lambda_0 s}\left[e^{-\frac{\lambda_1}{b_1}x_1e^{\lambda_1s}}-e^{-\frac{\lambda_1}{b_1}x_2e^{\lambda_1s}}\right].
\]
We conclude with the dominated convergence theorem, \eqref{conv:constant} and \eqref{eq:intersumQNi}, that
\begin{equation*}
    \frac{1}{N^{1-\alpha}}\sum_{i= \lfloor x_1N^{\lambda_1 t}\rfloor +1}^{\lfloor x_2N^{\lambda_1 t}\rfloor} Q(N,i) \underset{N\to\infty}{\longrightarrow}  \frac{\gamma b_0\omega\lambda_1}{b_1} \int_0^\infty (1+2b_0s)e^{-\lambda_0 s}\left[e^{-\frac{\lambda_1}{b_1}x_1e^{\lambda_1s}}-e^{-\frac{\lambda_1}{b_1}x_2e^{\lambda_1s}}\right]\, ds.
\end{equation*}
Recalling \eqref{eq:decSu} and \eqref{ineq:sumtRNi}, this ends the proof of Lemma~\ref{lem:expSuiNapp}.
\end{proof}

\section{Illustration and Discussion}
\label{sec:illus}

In this section, we illustrate and discuss the two main theorems, stated in Section~\ref{sec:model} and the four lemmas, presented in Section~\ref{sec:SFS}. In the following, we approximate by simulation the quantities for $N=500$, $t_N=\log(N)/\lambda_0$ and using the following set of parameters (called reference parameters set below),
\begin{center}
     $b_0=1.2$, $d_0=2.0$, $b_1=1.2$, $d_1=0.5$, $\omega=2.0$, $\alpha=0.9$ and $\gamma=1.0$.
\end{center}

As seen in Section~\ref{sec:SFS}, the SFS of a rescued population can be computed by counting two types of mutations separately: (1)the number of mutations that appear during sensitive cell division $\Su^N_i(t_N)$ and (2)those of resistant cell division $ \Sb^N_i(t_N)$.\\

Let us first deal with the case of mutations carried by a small number of cells (see Lemmas~\ref{lem:Sbieq} and \ref{lem:Sui}). In such case, the expected number of mutations is well approximated by the expected number of mutations appeared in resistant cells, i.e.
$$ \E\left[S^N_i(t_N)\right] \underset{N\to +\infty}{\sim} \E\left[\Sb^N_i(t_N)\right]. 
$$
We illustrate this equivalence relation in Figure~\ref{fig:Si} where the orange bullets and the red bars correspond to the empirical expectation over $50 \,000$ realizations, respectively, of $S^N_i(t_N)$ and $\Sb^N_i(t_N)$, for $i\in [0,20]$ (subfigure a) and $i\in [21,121]$ (subfigure b). The blue crosses correspond to their theoretical approximation given by Theorem~\ref{mainth} and Lemma~\ref{lem:Sbieq}, respectively. We observe that, for small values of $i$, the empirical expectations of $ S^N_i(t_N)$ and $\Sb^N_i(t_N)$ correspond, as do their theoretical approximations.
\begin{figure}[ht!]
  \centering
  \subcaptionbox{\centering $i\in [1,20]$.}{%
\includegraphics[width=0.9\textwidth]{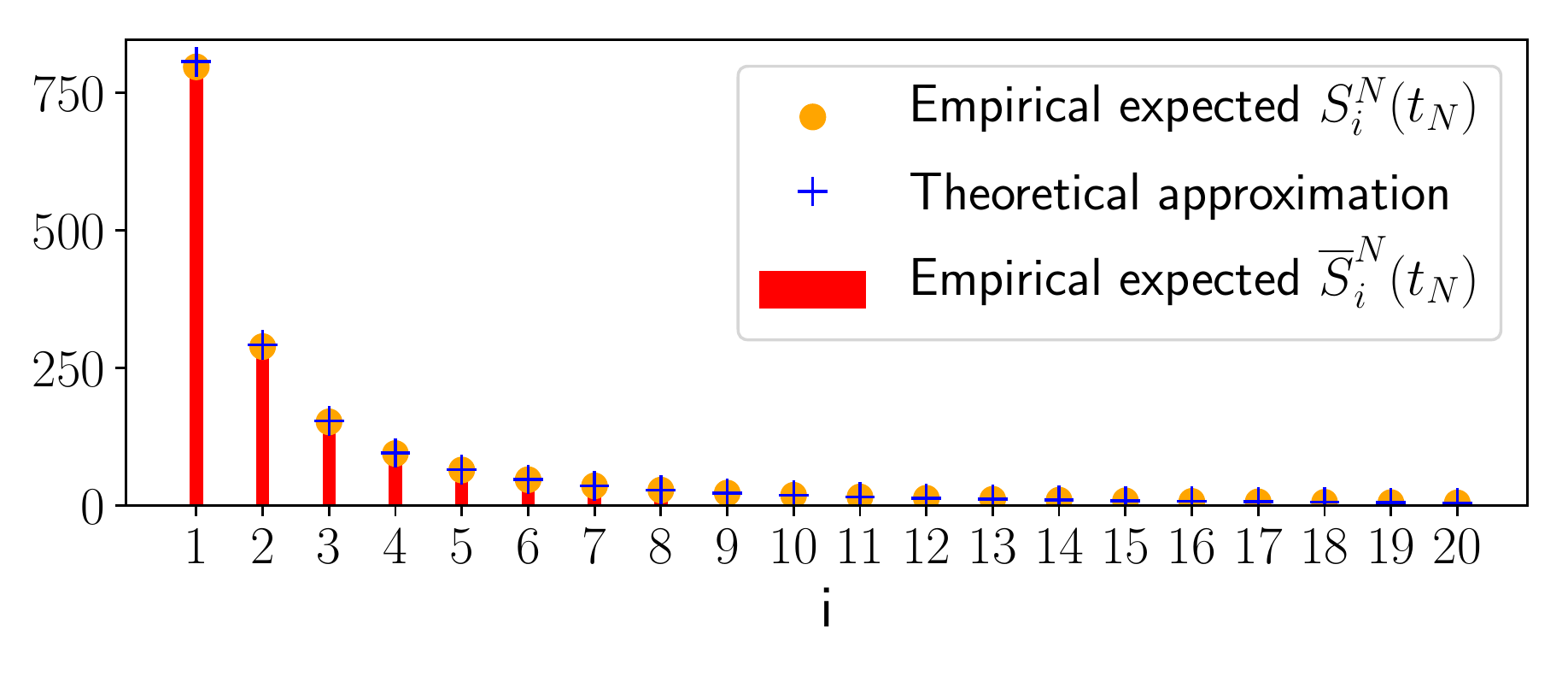}}\\
  \subcaptionbox{\centering $i\in [21,121]$.}{%
\includegraphics[width=0.9\textwidth]{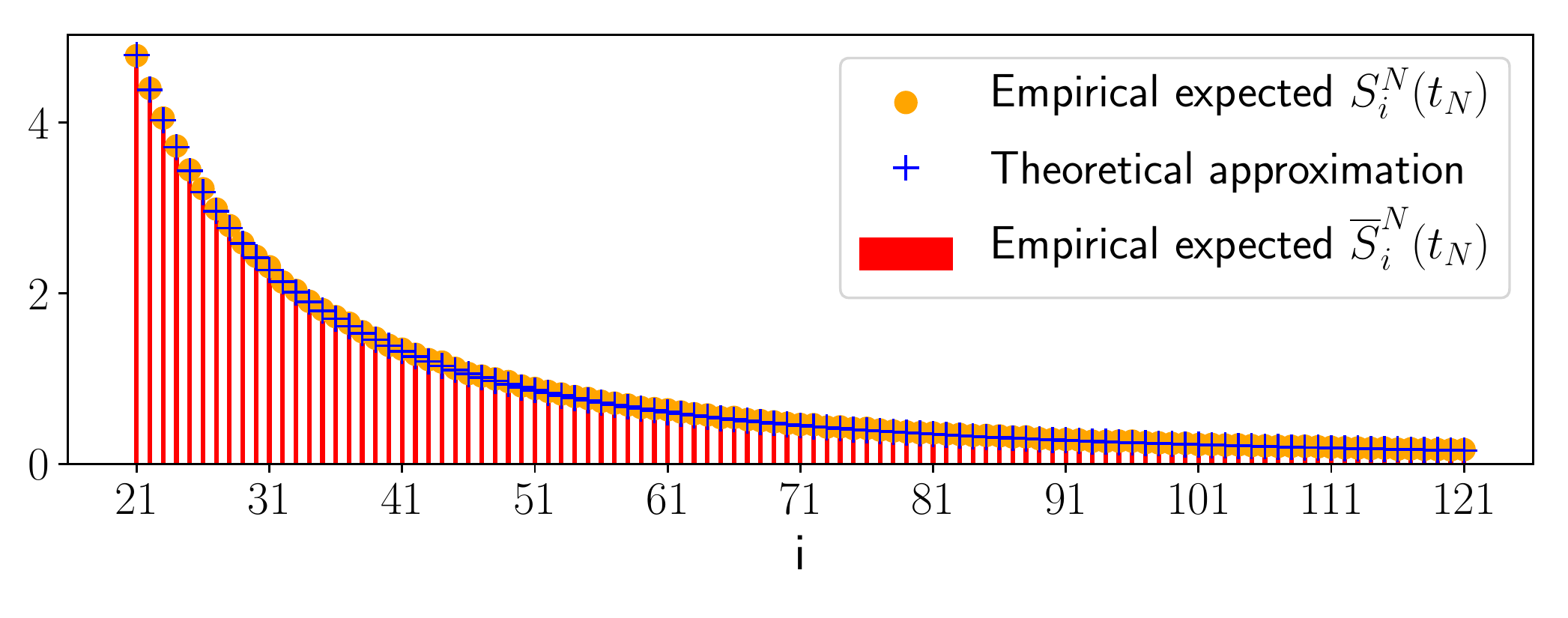}}
  \caption{Empirical and theoretical expectation of $S^N_i(t_N)$ and $\Sb^N_i(t_N)$ for small value of $i$, i.e $i\in [1,20]$ in (a) and $i\in [21,121]$ in (b). The orange bullets and the red bars correspond to the empirical expectation over $50 \,000$ realizations, respectively, of $S^N_i(t_N)$ and $\Sb^N_i(t_N)$. The blue crosses correspond to the theoretical approximation given by Theorem~\ref{mainth} and Lemma~\ref{lem:Sbieq}.}\label{fig:Si}
\end{figure}

Approximation given by Theorem~\ref{mainth} can be seen as the contribution of each \textit{ancestral resistant} cell, appeared at a random times of law given in Proposition~\ref{th:ancestralresistant} (ii), multiplied by the total number of \textit{ancestral resistant} cells that appeared before $t_N$,
\begin{equation}\label{eq:decomposition}
    \E\left[\Sb^N_i\big(t_N\big)\right] \underset{N\to \infty}{\sim} \underbrace{\frac{2b_0\gamma }{\lambda_0} N^{1-\alpha} }_{\substack{\text{number of}\\ \textit{ancestral}\\ \textit{resistant cells}}} 
    \underbrace{\frac{\omega\lambda_0}{\lambda_1+\lambda_0} I(i)N^{\lambda_1 t}}_{\substack{\text{contribution of a cell}}}.
\end{equation}
Indeed, in the one hand, the expected number of \textit{ancestral resistant} cells, $\E[N^N_R]$, is given by Lemma~\ref{lem:nbrresistant} and its asymptotic equivalent, when $N$ is large, corresponds to the underlined formula with the indication "number of \textit{ancestral resistant} cells" of Equation~\eqref{eq:decomposition}.  Notice that this approximation when $N$ is large is the same as the approximation of $N\P\left(\A_1^N\right)$ (see Formula~\eqref{eq:1anc}), which represents the mean number of \textit{ancestral resistant} cells unrelated to another \textit{ancestral resistant} cell.

On the other hand, recall that the time of occurrence of an \textit{ancestral resistant} cell conditioned on belonging to a progeny that carried exactly one \textit{ancestral resistant} cell follows an exponential law of parameter $\delta_0\, x_N$ (cf Proposition~\ref{th:ancestralresistant}). Then notice that $\delta_0 \, x_N$ tends to $\lambda_0$ when N tends to infinity. Moreover, when $i$ is fixed and $N$ is large, $\E \left[S_i^N(t_N-s)\vert Z^N(s)=(0,1) \right]$ can be approximated by $\omega I(i) e^{\lambda_1 (t_N-s)}$. Hence the second term of the product \eqref{eq:decomposition} can be interpreted as the SFS associated with the process generated by one \textit{ancestral resistant} cell appeared at an exponential time of parameter~$\lambda_0$, since
\begin{align*}
\int_0^{t_N} \E \left[S_i^N(t_N-s)\vert Z^N(s)=(0,1) \right] \lambda_0 e^{-\lambda_0 s} \, ds &\underset{N\to\infty}{\sim}
\int_0^{t_N}\omega I(i) e^{\lambda_1 (t_N-s)} \lambda_0 e^{-\lambda_0 s} \, ds \\
&\underset{N\to\infty}{\sim}
\omega \frac{\lambda_0}{\lambda_1+\lambda_0} I(i)N^{\lambda_1 t}.\\
\end{align*}

In conclusion, in this asymptotic case, the shape of the SFS with respect to $i$ is not impacted by the rescue dynamics. However, the expected SFS is impacted by the time it takes for the  \textit{ancestral resistant} cells to appear, although the SFS is studied at an asymptotically long time $t_N$.

The fraction $ \frac{\lambda_0}{\lambda_1+\lambda_0} $ can be interpreted as a loss coefficient due to the rescued dynamics. When $\lambda_0$ is large, the process $Z^N_0$ is extinct quickly, so the exponential time is close to $0$ and the loss coefficient is close to $1$. When $\lambda_1$ is large, even if the exponential time is small, starting the birth and death process induced by an \textit{ancestral resistant} cell from this time and not from $t=0$  is a huge disadvantage, so the loss coefficient is close to $0$.

The size order of $\E[\,\Sb^N_i\big(t_N\big)]$ is given by the size order of the resistant population number at time $t_N$ which is $N^{\lambda_1 t}$. The mutation number due to sensitive division can not reach such size order. This is why, it is negligible in such asymptotic case.\\

\medskip
Let us now illustrate and discuss the results on the number of mutations carried by a large number of resistant cells, i.e. for $i$ depending on $N$. Remind that in this case, we assumed that $i_N$ is proportional to the size order of the resistant population number at time $t_N$, i.e $i_N \,{\sim} \,e^{\lambda_1 t_N}$ when $N\to\infty$. As in Formula~\eqref{eq:decomposition}, the approximation of $\E[\Sb_{i_N}^N(t_N)]$, given by Lemma~\ref{lem:Sbxeq}(i), can be seen as the expected number of \textit{ancestral resistant} cells multiplied by the contribution of one \textit{ancestral resistant} cell, 

\begin{equation}\label{eq:interSb}
b_0\gamma \omega \lambda_1 \, K'(x)N^{1-\lambda_1 t -\alpha}= \underbrace{\frac{2b_0\gamma }{\lambda_0} N^{1-\alpha} }_{\substack{\text{number of}\\ \textit{ancestral}\\ \textit{resistant cells}}}  \int_0^\infty \underbrace{ \left(\int_0^u \omega b_1 e^{\lambda_1(u-s)} \lambda_0 e^{-\lambda_0 s}ds\right)}_{\substack{\Omega_1}} \, \underbrace{(\frac{\lambda_1}{b_1})^2 e^{-\lambda_1(t_N-u)}e^{-x\frac{\lambda_1}{b_1}e^{\lambda_1u}}}_{\substack{\Omega_2}} du .
\end{equation}
The integral term $\Omega_1$ represents the number of mutations, due to resistant division, that appeared at time $u$. Indeed, $\omega$ represents the mean number of mutations due to one division, $b_1 \, ds$ represents the probability that a resistant cell divides during a small interval of time $ds$ and $e^{\lambda_1(u-s)}$ represents the number of progeny at time $u$ of an \textit{ancestral resistant} cell appeared at time $s$. Finally, the integration over an exponential density of parameter $\lambda_0$ is due to the stochastic time of occurrence of each \textit{ancestral resistant} cell, following a law given by Proposition~\ref{th:ancestralresistant}(ii). $\Omega_2$ gives the approximation when $N$ tends to infinity, of the probability that a resistant cell appeared at time $u$ has exactly $i_N$ progeny at time $t_N$ (see Formula \eqref{eq:probaZi} and convergence result \eqref{conv:lN}).

As previously explained, we are not able to give an asymptotic approximation of $\E \left[\Su_{i_N}^N(t_N)\right]$ as we are not able to control the influence of the kinship events between \textit{ancestral resistant} cells. However, in Figure~\ref{fig:SiN}, the empirical expectation of $\Su_{i_N}^N(t_N)$, $S_{i_N}^N(t_N)$ and $\Sb_{i_N}^N(t_N)$ for $i_N \in [200,700]$ are respectively represented by the green line, the orange bullets and the pink line. Considering our reference parameters set, notice that $e^{\lambda_1 t_N} \approx 230$. We observe that the empirical expectations of $\Su_{i_N}^N(t_N)$ and $\Sb_{i_N}^N(t_N)$ have the same order of magnitude when $i_N \,{\sim}\, e^{\lambda_1 t_N}$ when $N\to\infty$. We observe also that, when $i_N$ increases, the empirical expectation of $\Su_{i_N}^N(t_N)$ decreases faster to $0$ than the one of $\Sb_{i_N}^N(t_N)$.\\
\begin{figure}[!ht]
  \centering
    \includegraphics[width=1.0\textwidth]{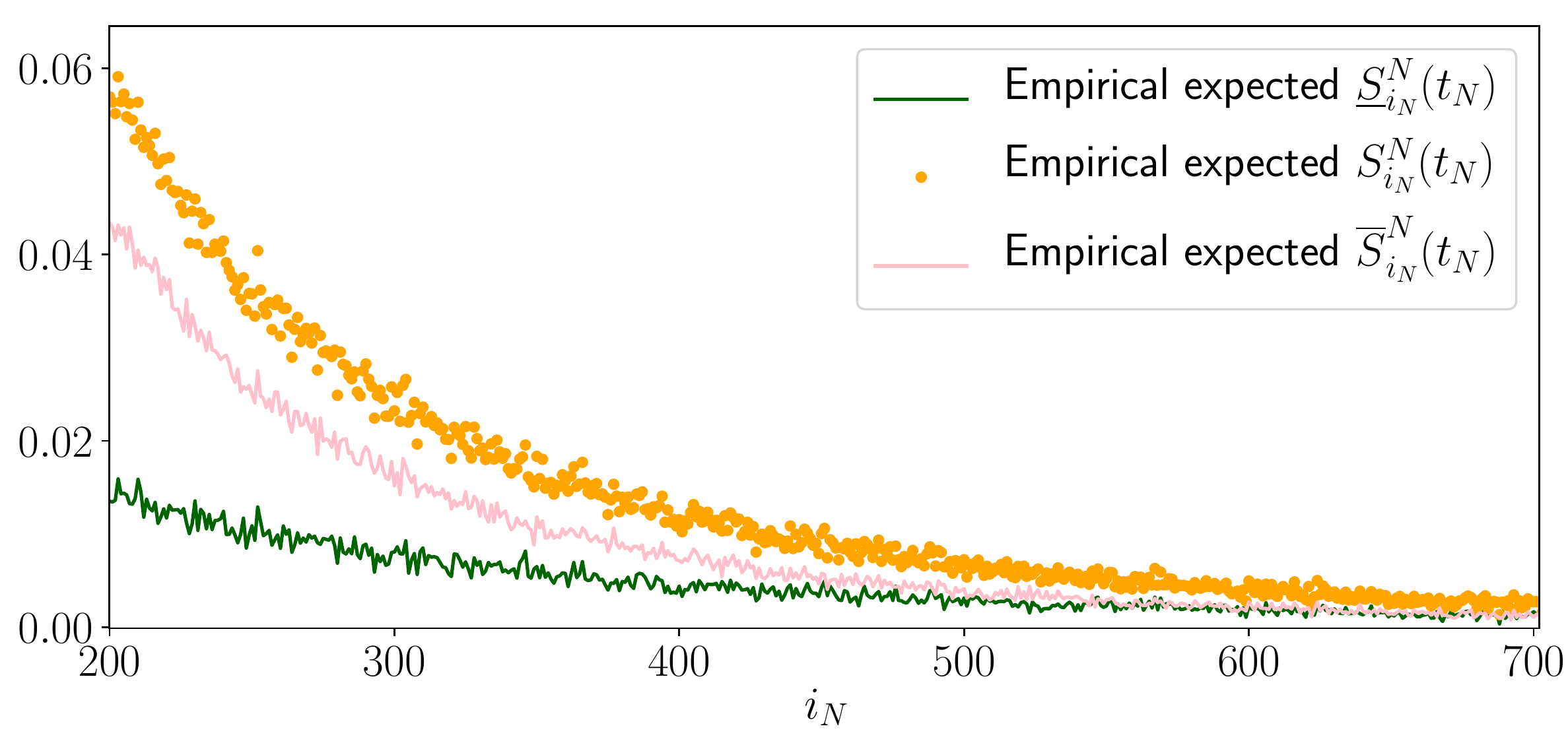}
  \caption{Empirical expectation of $\Su_{i_N}^N(t_N)$, $S_{i_N}^N(t_N)$ and $\Sb_{i_N}^N(t_N)$ for large values of $i_N$, i.e $i_N \in [200,700]$. The orange bullets, the pink and green lines correspond to the empirical expectation over $50 \,000$ realizations, respectively, of $S^N_{i_N}(t_N)$, $\Sb^N_{i_N}(t_N)$, and $\Su^N_{i_N}(t_N)$.}\label{fig:SiN}
\end{figure}


\medskip 
Finally, we illustrate and discuss the result we obtain on $S_{x_1,x_2}(t_N)$ with $0<x_1<x_2$, defined in~\eqref{def_Sx1x2}. To simplify the understanding of the illustrations and discussions, we focus on $S_x(t_N):= S_{x,+\infty}(t_N)$, for $x>0$.

We first deal with the results about $\Su_{x}^N(t_N)$, whose empirical expectation, using our reference parameters set, is drawn in Figure~\ref{fig:Sxu} with orange bullets. The blue line, in Figure~\ref{fig:Sxu}, corresponds to the function $x\mapsto b_0\gamma\omega \lambda_1 \,L(x)\,N^{1-\alpha} $, which is the theoretical approximation of $\E[\Su_{x}^N(t_N)]$ when $N$ is large (see Lemma~\ref{lem:expSuiNapp}).
\begin{figure}[!ht]
  \centering
    \includegraphics[width=0.8\textwidth]{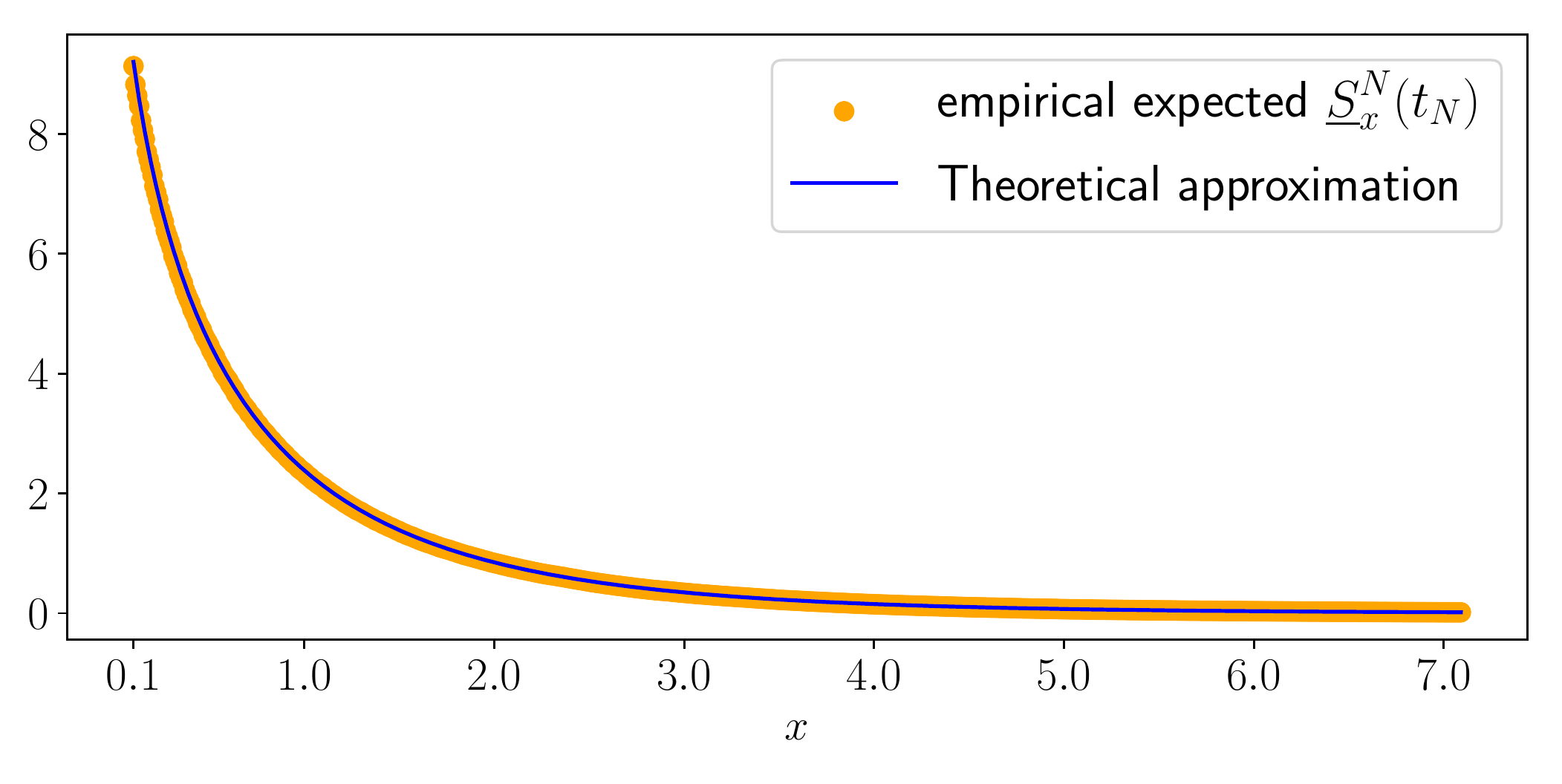}
  \caption{Empirical expectation of $\Su_{x}^N(t_N)$ for $x \in [0.1,7]$. The orange bullets correspond to the empirical expectation over $50 \,000$ realizations of $\Su_{x}^N(t_N)$. The blue line corresponds to the function ${b_0\gamma \omega \lambda_1} \, L\, N^{1-\alpha}$ with $L$ defined in \eqref{def:L}.}\label{fig:Sxu}
\end{figure}
Let us discuss the shape of this theoretical approximation. Recall that, for~$x>0$
\begin{equation}\label{eq:interSu}
N^{1-\alpha} b_0\gamma \omega \lambda_1 \, L(x)= \underbrace{\frac{2b_0\gamma }{\lambda_0} N^{1-\alpha} }_{\substack{\text{number of}\\ \textit{ancestral}\\ \textit{resistant cells}}}  \int_0^\infty \underbrace{ \frac{\omega}{2}(1+2b_0 s)}_{\substack{\Delta_1}} \, \underbrace{\frac{\lambda_1}{b_1}e^{-x\frac{\lambda_1}{b_1}e^{\lambda_1s}}}_{\substack{\Delta_2}}\underbrace{\lambda_0e^{-\lambda_0s}}_{\substack{\Delta_3}} ds .
\end{equation}
As previously discussed, the first term is an approximation of the expected total number of \textit{ancestral resistant} cells. Then, $\Delta_3$ corresponds to the limiting density of time of occurrence of an \textit{ancestral resistant} cells (cf Proposition~\ref{th:ancestralresistant}). In view of~\eqref{eq:probaZi} and the proof of Lemma~\ref{lem:expSuiNapp}, we can interpret $\Delta_2$ as an approximation of the probability that a cell, appeared at time $s$, has more than $xe^{\lambda_1 t_N}$ offspring at time $t_N$. Indeed,
\begin{align*}
    \P(\tilde{Z}(t_N-s)> x e^{\lambda_1 t_N})&=\sum_{i> x e^{\lambda_1 t_N}}\frac{\lambda_1^2}{b_1^2}\frac{e^{-\lambda_1(t_N-s)}}{(1-d_1e^{-\lambda_1(t_N-s)}/b_1)^2}\left(\frac{1-e^{-\lambda_1(t_N-s)}}{1-d_1e^{-\lambda_1(t_N-s)}/b_1}\right)^{i-1}\\
    &=\frac{\lambda_1}{b_1} \frac{e^{-\lambda_1 u}}{e^{-\lambda_1u}-\frac{d_1}{b_1}e^{-\lambda_1 t_N}}\left(\frac{1-e^{-\lambda_1(t_N-s)}}{1-d_1e^{-\lambda_1(t_N-s)}/b_1}\right)^{\lfloor x e^{\lambda_1 t_N}\rfloor}\underset{N\to\infty}{\longrightarrow} \Delta_2
\end{align*}
Finally, the factor $\Delta_1=\frac{\omega}{2}(1+2b_0 s)$ represents the mean number of mutations carried by a resistant cell that appeared at time $s$, which is proportional to the number of times a sensitive cell divides before becoming resistant at time $s$. Indeed, the term ${\omega}/{2}$ corresponds to the mean number of mutations that a cell gets after one division. The term $(1+2b_0s)$ corresponds to the division needed to become resistant and to the mean number of divisions a resistant cell makes before appearing at time $s$ knowing that its family tree has only one resistant cell. The mean number of divisions made by a sensitive cell over a time $[0,s]$ is given by $b_0s$. Hence the factor $2$ is surprising. However it has already been met in the dynamics of branching processes (see remark of the main theorem in \cite{Bansaye}). Here, it translates the increase of the number of divisions in a resistant lineage compared to a sensitive one (destined to die out).\\

Let us then discuss results about $\Sb_{x}^N(t_N)$. 
In Figure~\ref{fig:Sxb}, we draw the empirical expectation of this quantity  using orange bullets.
The blue line represents the function $x\mapsto  b_0 \gamma \omega \lambda_1 N^{1-\alpha} K(x)$ that is the asymptotic approximation of  $\E[\Sb_{x}^N(t_N)]$ when $N$ is large, given by Lemma~\ref{lem:Sbxeq}.
\begin{figure}[!ht]
  \centering
\includegraphics[width=0.85\textwidth]{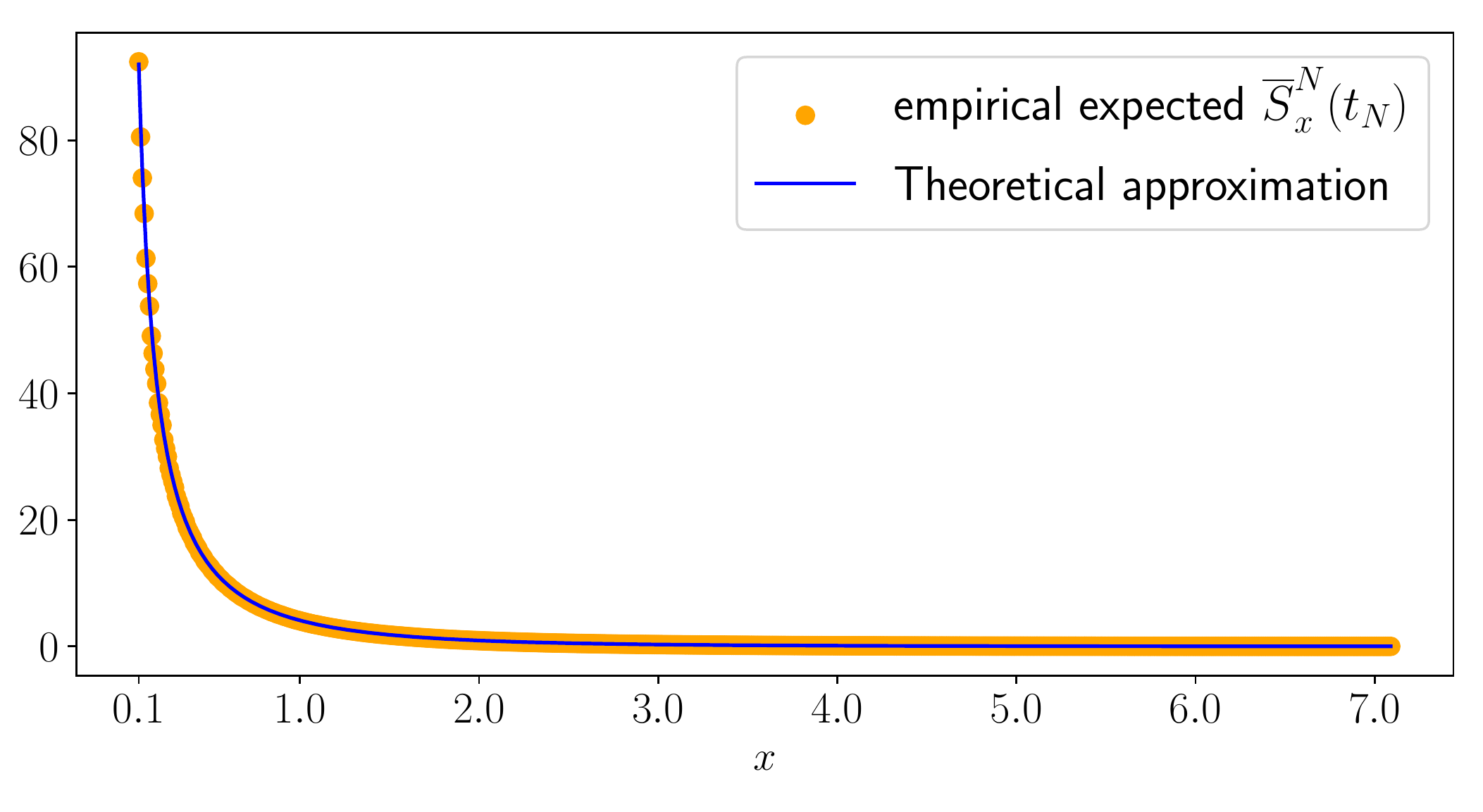}
  \caption{Empirical expectation of $\Sb_{x}^N(t_N)$ for $x \in [0.1,7]$. The orange bullets correspond to the empirical expectation over $50 \,000$ realizations of $\Sb_{x}^N(t_N)$. The blue lines correspond to the function $ b_0 \gamma \omega \lambda_1 N^{1-\alpha} K$ with $K$ defined in \eqref{def:K}.}\label{fig:Sxb}
\end{figure}

Let us discuss the shape of this theoretical approximation. Recall that, for~$x>0$
\begin{equation}\label{eq:interSux}
 b_0\gamma \omega \lambda_1 K(x) N^{1-\alpha}= \underbrace{\frac{2b_0\gamma }{\lambda_0} N^{1-\alpha} }_{\substack{\text{number of}\\ \textit{ancestral}\\ \textit{resistant cells}}}  \int_0^\infty \underbrace{ \left(\int_0^u \omega b_1 e^{\lambda_1(u-s)} \lambda_0 e^{-\lambda_0 s}ds\right)}_{\substack{\Omega_1}} \, \underbrace{\frac{\lambda_1}{b_1}e^{-x\frac{\lambda_1}{b_1}e^{\lambda_1u}}}_{\substack{\Delta_2}} du.
\end{equation}
This expression is the same as the approximation of $\E \left[\Su_{i_N}^N(t_N)\right]$ given by \eqref{eq:interSb} for which we replace $\Omega_2$ by $\Delta_2$, defined in \eqref{eq:interSu}. Indeed, in this case, we take into account mutations carried by more than $xe^{\lambda_1 t_N}$ cells and not exactly $i_N \sim e^{\lambda_1 t_N}$. For this reason, $\Omega_2$, as a function of $x$, corresponds to the derivative of $\Delta_2$ correctly renormalized.


Finally, in Figure~\ref{fig:Compare}, we draw using dash and continuous lines respectively, the functions $K$ and $L$ on $[0.6,6]$, which represents the weight of the contribution of mutations due to resistant and sensitive divisions respectively (see Formula~\eqref{eq:expSxapp}). 
\begin{figure}[!ht]
  \centering
  \subcaptionbox{\centering $b_0=1.2$, $\lambda_0=0.8$.}{%
    \includegraphics[width=0.5\textwidth]{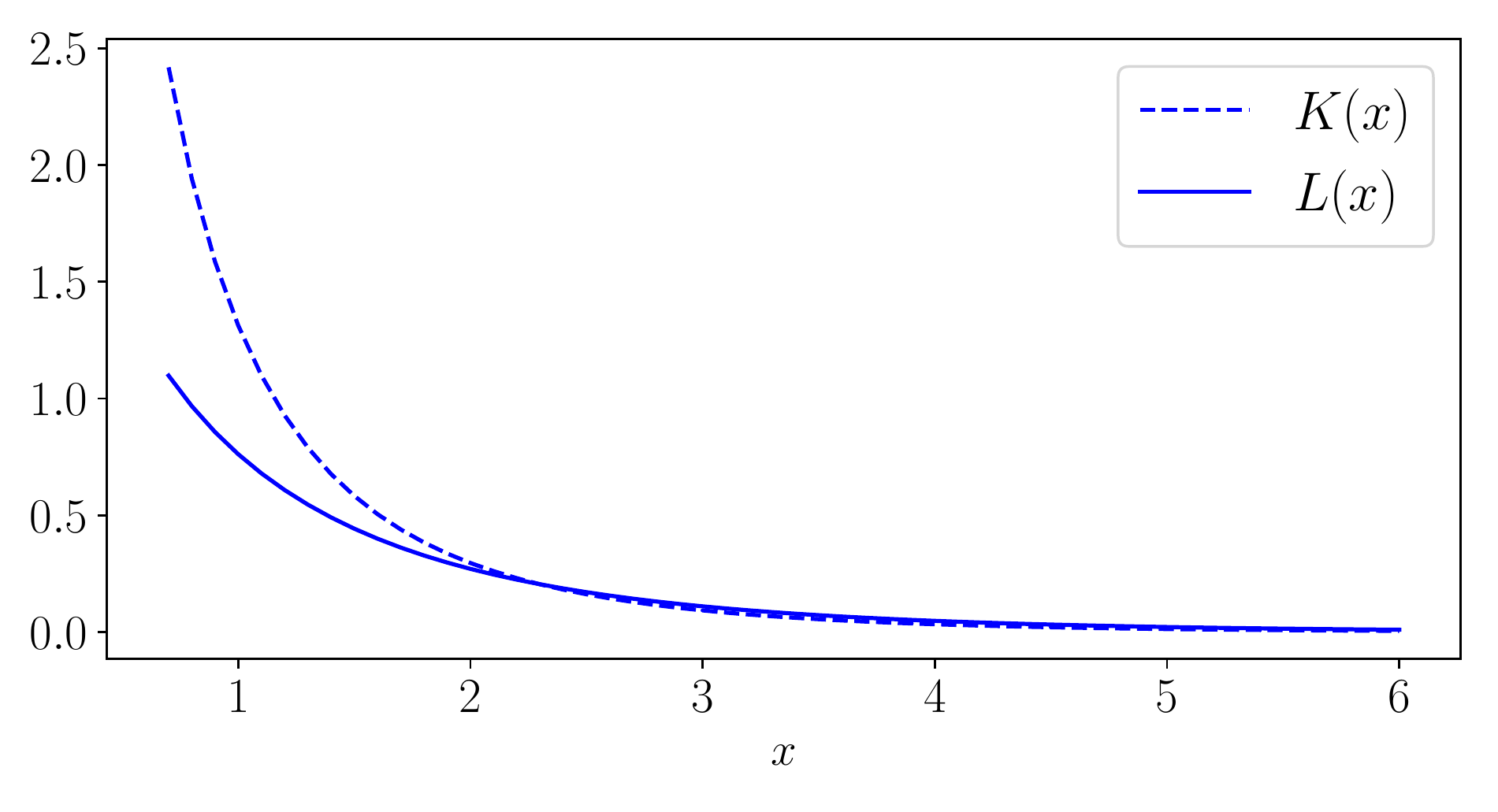}}\subcaptionbox{\centering $b_0=2.2$, $\lambda_0=0.8$.}{%
    \includegraphics[width=0.5\textwidth]{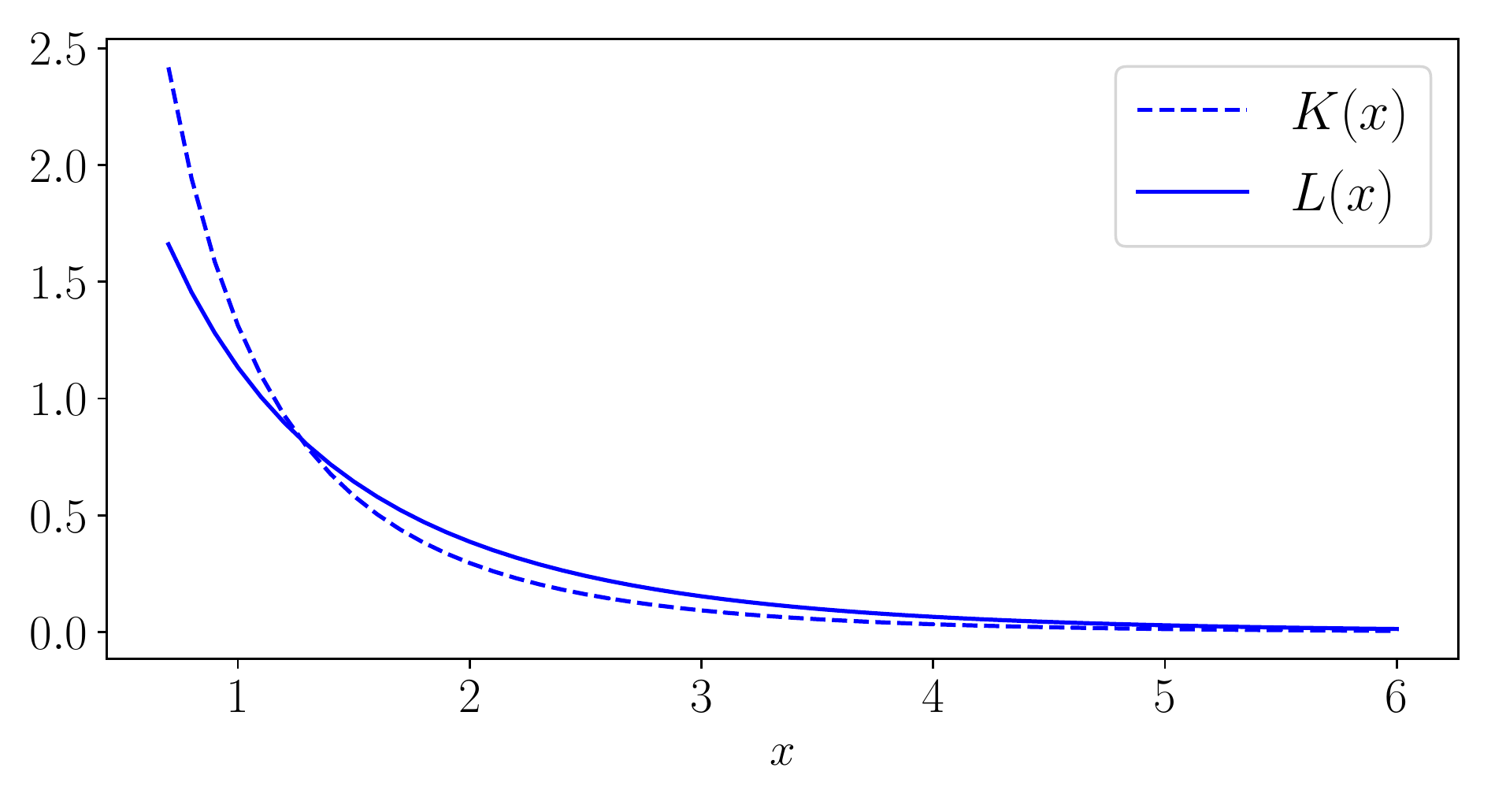}}
  \subcaptionbox{\centering $b_0=1.2$, $\lambda_0=0.3$. }{%
    \includegraphics[width=0.5\textwidth]{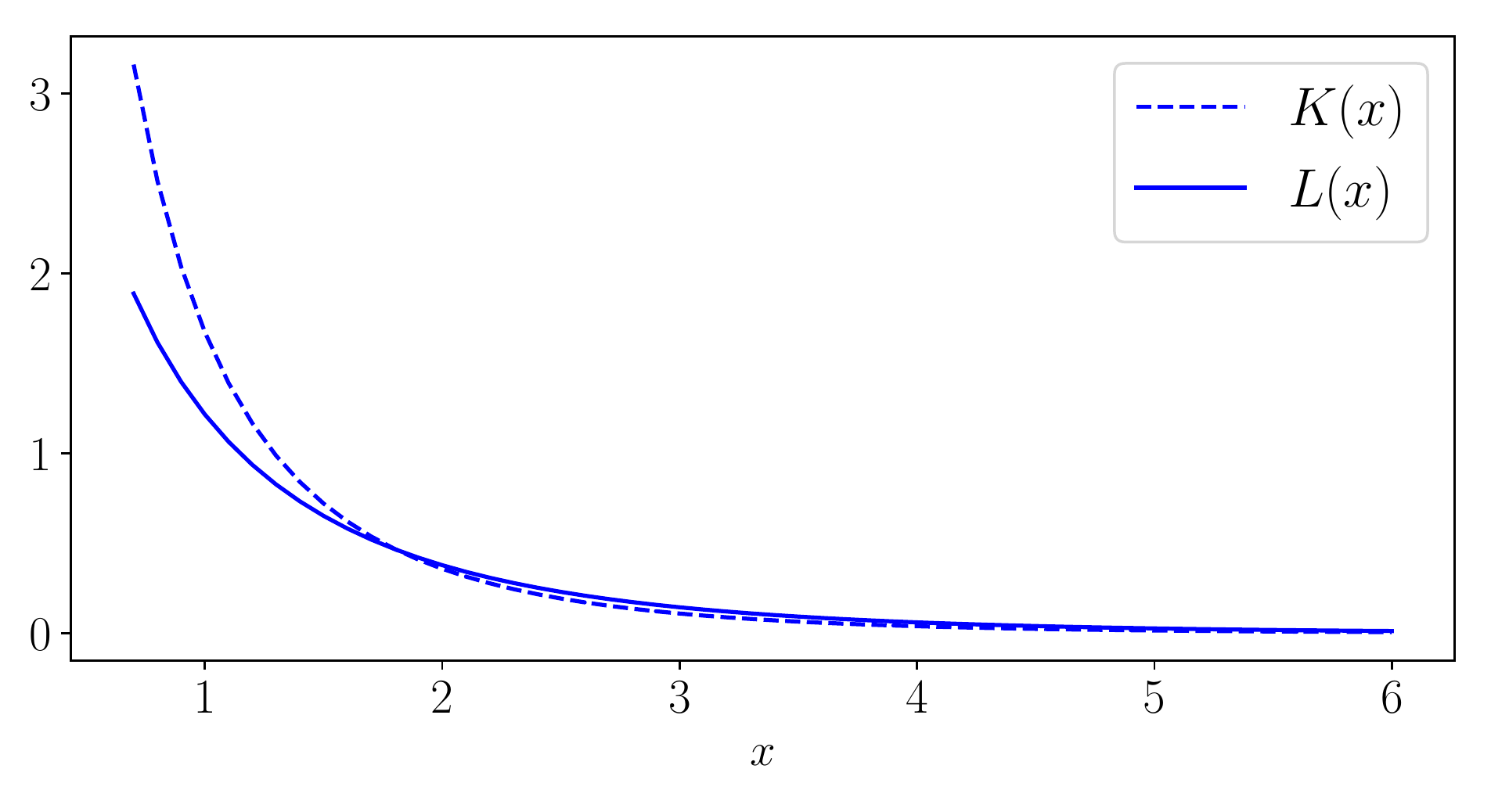}}\subcaptionbox{\centering $b_0=2.2$, $\lambda_0=0.3$}{%
    \includegraphics[width=0.5\textwidth]{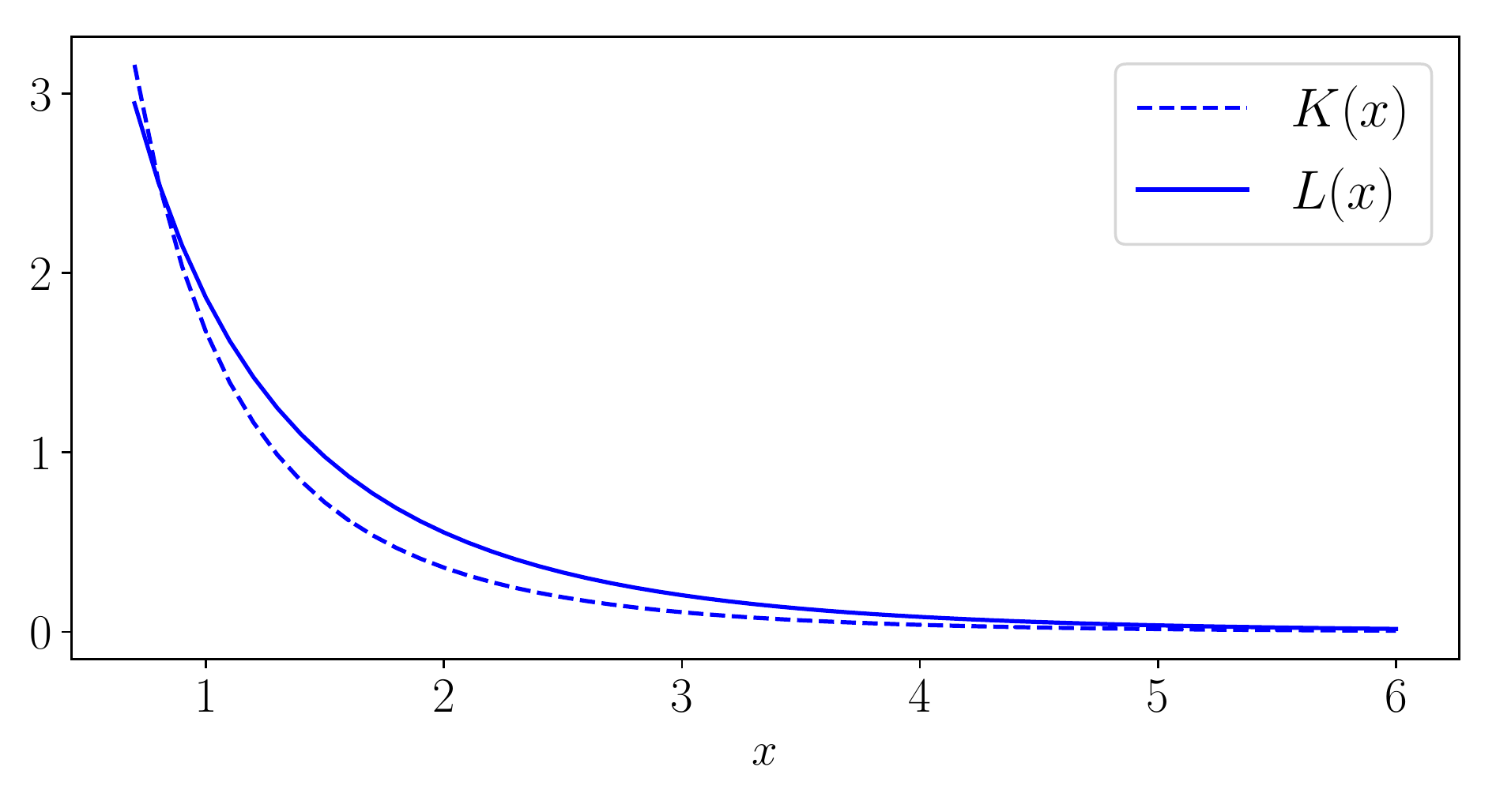}}
  \caption{Graph of the functions $K$ and $L$ on $[0.6,6]$ for the reference parameter set with different values of $b_0$ and $\lambda_0$. The dash line represents the function $K(x)$ for $x\in [0.6,6]$, defined in \eqref{def:K}. The continuous line represents the function $L(x)$ for $x\in [0.6,6]$, defined in \eqref{def:L}.}\label{fig:Compare}
\end{figure}
Notice that, depending on the value of the birth and growth rates of sensitive cells, the contribution $\E[\Su_x^N(t_N)]$ of mutations appeared in sensitive cells may become larger than the contribution $\E[\Sb_x^N(t_N)]$ of mutations appeared in resistant cells, this is true for example when $b_0$ is large, $\lambda_0$ is small and $x$ is sufficiently large ($x>1$ on Subfigure (d)). Moreover we observe that, when we don't take into account the number of \textit{ancestral resistant} cells, the influence of the sensitive dynamics in the contribution of mutations, that appeared during resistant divisions, is only due to $\lambda_0$ contrary to the mutations appeared during sensitive divisions.

\medskip

We conclude that the rescue dynamics influence the SFS associated with mutations carried by both a small and a large number of cells in the population at the characteristic time $t_N$ of extinction of sensitive cells. \\

\section*{Acknowledgements}

This research was led with financial support from ITMO Cancer of AVIESAN (Alliance Nationale pour les Sciences de la Vie et de la Santé, National Alliance for Life Sciences \& Health) within the framework of the Cancer Plan.

\appendix

\section{Annexe}

The following lemma gives the expected number of resistant mutations that occur during the extinction of the sensitive initial population.
\begin{Lem}
\label{lem:nbrresistant}
Let $N^N_R$ be the total number of \textit{ancestral resistant} cells in the process described in Section~\ref{sec:model}. 
\begin{equation}\label{eq:equivnbrmut}
    \E[N^N_R]=\frac{2b_0\gamma_N}{\lambda_0+2\gamma_Nb_0} N \underset{N\to\infty}{\sim} \frac{2b_0\gamma}{\lambda_0}N^{1-\alpha}
\end{equation}
\end{Lem}

\begin{proof}
The dynamics of the process $N^N_R$ are influenced by only two different events which can be modeled by two independent Poisson processes. Hence, there exist $Y_1$ and $Y_2$, two independent Poisson processes which are both independent of the process $Z_0$, such that we can write
$$N^N_R=Y_1 \left( 2b_0\gamma_N(1-\gamma_N)\int_0^{\infty} Z^N_0(s)ds\right) +2 \, Y_2 \left( b_0\gamma_N^2\int_0^{\infty} Z^N_0(s)ds\right).$$

Notice that, for all $t \geq 0$, $\mathbb{E}[Z^N_0(t)]= Ne^{-(\lambda_0+2\gamma_Nb_0)\,t}$, we deduce by taking the expectation in the previous expression that \begin{equation*}
\mathbb{E}[N_R^N]= 2b_0\gamma_N(1-\gamma_N)\int_0^{\infty} \mathbb{E}[Z^N_0(s)]ds+2 b_0\gamma_N^2\int_0^{\infty} \mathbb{E}[Z^N_0(s)]ds = \frac{2b_0 \gamma_N}{\lambda_0+2\gamma_Nb_0}N.
\end{equation*}
\end{proof}

This second lemma states the SFS of a birth and death process that is supercritical and that starts with one individual. The proof follows the proof of Proposition 3.1 in \cite{Richard} and is given for clarity in order to adapt it to our notation.
\begin{Lem}
\label{lemS1}
    For all $i\in \N^*$, $N\in \N^*$ and $t> 0$,
    \begin{equation*}\label{eq:SFS1}
    \mathbb{E}\left[S^{N}_i(t) \vert Z^N(0)=(0,1)\right]= w e^{\lambda_1 t} \int_0^{ \frac{e^{\lambda_1 t}-1}{e^{\lambda_1 t}-d_1/b_1}} \frac{ 1-y}{ 1-\frac{d_1}{b_1}y}y^{i-1}dy.
    \end{equation*}
    Moreover,  for $t_N=t\log(N)$
    \begin{equation}\label{Gunnarsson}
    \mathbb{E}\left[S^{N}_i(t_N) \vert Z^N(0)=(0,1)\right]\underset{N\to\infty}{\sim} w N^{\lambda_1 t} I(i),
    \end{equation}
    where the definition of $I(i)$ is given by~\eqref{def:I}.
\end{Lem}
\begin{proof}
We assume that initially the population is given by only one resistant cell. We denote by $M^{i,da}_t$ the number of mutations carried by $i$ resistant cells at time $t$ with ages in $[a-da,a]$.
Hence $M^{i,da}_t$ is given by the number of mutations that appeared on any resistant cells at time $t-a$,
	\begin{equation*}
	\mathbb{E}[M_t^{i,da}]= 2 \, \mathbb{E} \Big[ \sum_{j=1}^{Z_1(t-a)} N_w\one_{\text{division between $t-a-da$ and $t-a$, } \tilde{Z}^j_1(a)=i}\Big] 
	\end{equation*}
	where $\tilde{Z}^j_1$ is a supercritical birth and death process with parameters $b_1>d_1$ starting by one cell. The $2$ in the previous formula corresponds to the contribution of the two daughter cells using the branching property of the process we are studying.
 
	Then taking conditional expectation and using independence between cells we obtain, 
	\begin{equation*}
	\begin{split}
	\mathbb{E}[M_t^{i,da}]&= 2 \, \mathbb{E} \Big[ Z_1(t-a) \Big] \mathbb{E} \Big[ N_w\Big] \mathbb{P}(\text{division between $t-a-da$ and $t-a$}) \mathbb{P}(Z_1(a)=i)\\
	&=2\,e^{(b_1-d_1)(t-a)}(w/2) b_1 da \mathbb{P}(\tilde{Z}_1(a)=i).
	\end{split}
	\end{equation*}
	For all $t>0$, the law of $\tilde{Z}_1(t)$ is given by (see \cite{Harris1963} Chapt 5; formula (7.3)),
	$$
	\mathbb{P}(\tilde{Z}_1(a)=i)= \frac{(b_1-d_1)^2 e^{(b_1-d_1)a}}{(b_1e^{(b_1-d_1)a}-d_1)^2} \Big( \frac{b_1(e^{(b_1-d_1)a}-1)}{b_1e^{(b_1-d_1)a}-d_1} \Big)^{i-1}
	$$
	Hence, we obtain, $$ \mathbb{E}[M_t^{i,da}]=w \frac{\lambda_1^2 e^{\lambda_1t}}{(b_1e^{\lambda_1 a}-d_1)^2} \Big( \frac{b_1(e^{\lambda_1 a}-1)}{b_1e^{\lambda_1 a}-d_1} \Big)^{i-1} b_1 da.$$
	The result is deduced by a change of variable as those found in \cite{Gunnarssonetal2021} and by noticing that $$\mathbb{E}[S^{N}_i(t) \vert Z^N(0)=(0,1)]=\int_0^t \mathbb{E}[M_t^{i,da}]da.$$ 
	
	Finally, the asymptotically equivalent expression is obtained as soon as we notice that \[
	\frac{e^{\lambda_1 t_N}-1}{e^{\lambda_1 t_N}-d_1/b_1}\underset{N\to\infty}{\longrightarrow} 1.
	\]
\end{proof}

Following the proof of Proposition~\ref{th:ancestralresistant}, the last lemma gives the law of the generation and the appearance time of one \textit{ancestral resistant} cell chosen uniformly at random conditioned on belonging to a progeny that carried at least one \textit{ancestral resistant} cell (and not exactly one as for Proposition~\ref{th:ancestralresistant}).
\begin{Lem}\label{lem:TN_annexe}
 For any $N\in \N^*$, we denote by $ \tilde{G}_N$ and $\tilde{T}_N$, the generation and the appearance time of an \textit{ancestral resistant} cell chosen uniformly at random and conditioned on belonging to a progeny that carried at least one \textit{ancestral resistant} cell and we introduce
	\begin{equation}
	\label{def:pNtilde}
 \tilde{p}_N:=\frac{1-x_N}{2}=\frac{1-\sqrt{1-4p_N(1-p_N)(1-\beta_N)}}{2}
	\end{equation} with $x_N$, $p_N$ and $\beta_N$ defined in \eqref{def:pNbetaN}. Then
	\begin{itemize}
		\item[(i)] the law of $\tilde{G}_N$ is characterized by, for all $g \in \mathbb{N^*}$, \begin{equation*}
		\mathbb{P}(\tilde{G}_N=g)= \frac{2^{g-1}}{p_N-\tilde{p}_N}\left[\frac{1}{g}((p_N)^{g}-(\tilde{p}_N)^{g})-\frac{2}{g+1}((p_N)^{g+1}-(\tilde{p}_N)^{g+1}) \right]
		\end{equation*}
	
		\item[(ii)] and the density of $\tilde{T}_N$, $f_{\tilde{T}_N}$, is written,  for all  $t\geq0$, 
		\begin{equation*}
		f_{\tilde{T}_N}(t) = \frac{ e^{-t(b_0+d_0)(1-2p_N)}}{2t\,(p_N-\tilde{p}_N)} \left[ (1-e_N(t))\left(1+\frac{1}{t(b_0+d_0)}\right) - 2(p_N-\tilde{p}_Ne_N(t))\right]
		\end{equation*}
		where $$e_N(t)=e^{-\displaystyle2t(b_0+d_0)\,(p_N-\tilde{p}_N)}.$$
	\end{itemize}

\end{Lem}
\begin{proof}
This proof follows the proof of Proposition~\ref{th:ancestralresistant} and will use Lemma~\ref{lemma:GWgen} for which the event $\mathcal{R}^{\mathcal{T}(p,\beta)}$ will not correspond to $\mathcal{T}(p,\beta)$ has exactly one marked leaf but at least one marked leaf. In the following we make the calculus corresponding to this change.
Notice that, in this case, for all $g,n\in \N^*$, $\P(\RR|G=g,\lambda(\T)=n)=\P(\RR|\lambda(\T)=n)=1-(1-\beta)^n$. Hence the Equation~\eqref{eq:HcondR1} become,
\begin{equation}\label{eq:PGR}
        \P(G=g|\RR)=\frac{\sum_{n\geq 1}(1-(1-\beta)^n)v_{g,n}}{\sum_{n\geq 1}(1-(1-\beta)^n)u_n}=\frac{\sum_{n\geq 1}v_{g,n}-\sum_{n\geq 1}(1-\beta)^nv_{g,n}}{1-\sum_{n\geq 1}(1-\beta)^nu_n}.
    \end{equation}
Thus, using the same arguments as previously and Equation~\eqref{eq:linksumalphap} which is true for any $p<1/2$, we can calculus the three sums of Equation~\eqref{eq:PGR}. Indeed
    \begin{equation}\label{eq:sum1_an}
    \sum_{n\geq 1}(1-\beta)^nu_n=\frac{1}{p}\sum_{n\geq 1}\alpha_n[(1-\beta)p(1-p)]^n
    =\frac{1}{p}\sum_{n\geq 1}\alpha_n\tilde{p}^n(1-\tilde{p})^n=\frac{\tilde p }{p},
    \end{equation}
    where $\tilde{p}$ satisfies  $\tilde{p}<1/2$ and \begin{equation}\label{def:tildep_an}
	\tilde{p}(1-\tilde{p})=p(1-p)(1-\beta).
	\end{equation}
 Moreover, using induction formula~\eqref{rec:v}, we obtain for all $g\geq3$, 
	\begin{align*}
	    \sum_{n\geq 1}v_{g,n}&=\frac{2^{g-1}}{p}\sum_{n\geq g}\frac{[p(1-p)]^n}{n}\gamma_{n,g}\\
	    &=\frac{2^{g-1}}{p}\int_0^p(1-2x)\sum_{n\geq g}[x(1-x)]^{n-1}\gamma_{n,g}\\
	    &=\frac{2^{g-1}}{p}\int_0^p\frac{(1-2x)}{x(1-x)}\sum_{n\geq g}\sum_{i=1}^{n-g+1}\alpha_i[x(1-x)]^{i}\gamma_{n-i,g-1}[x(1-x)]^{n-i},
	    \end{align*}
	    where we identify a Cauchy product of infinite sum. Thus, using again an induction and then Equation~\eqref{eq:linksumalphap}, we obtain for all $g\geq3$, 
	    \begin{align}
	   \sum_{n\geq 1}v_{g,n}&=\frac{2^{g-1}}{p}\int_0^p\frac{(1-2x)}{x(1-x)}\left(\sum_{n\geq 1} \alpha_n[x(1-x)]^{n}\right)\left(\sum_{n\geq g-1}\gamma_{n,g-1}[x(1-x)]^{n}\right)dx,\nonumber\\
	   &=\frac{2^{g-1}}{p}\int_0^p\frac{(1-2x)}{x(1-x)}\left(\sum_{n\geq 1} \alpha_n[x(1-x)]^{n}\right)^{g-2}\left(\sum_{n\geq 2}\alpha_{n-1}[x(1-x)]^{n}\right)dx,\nonumber\\
	   &=\frac{2^{g-1}}{p}\int_0^p(1-2x)x^{g-1}dx,\nonumber\\
	   &=(2p)^{g-1}\left(\frac{1}{g}-\frac{2p}{g+1}\right),\label{eq:sum2_an}
	\end{align}
	which gives the value of the second sum. Moreover, as $\sum_{n\geq 1}v_{1,n}=1-p$ and \\ $\sum_{n\geq 1}v_{2,n} = \frac{2}{p}\int_0^p(1-2x)x \, dx$, Formula \eqref{eq:sum2_an} is also true for $g \in \{1,2\}$. To deal with the last sum of the r.h.s. of Equation~\eqref{eq:PGR}, we use the fact that the previous computations to deal with the second sum are true for any $p<1/2$. Thus, writing $\tilde{v}_{g,n}$ such that $v_{g,n}=p^{n-1}(1-p)^n\tilde{v}_{g,n}$ and using Equations~\eqref{def:tildep_an} and~\eqref{eq:sum2_an}, we have
    \begin{align}
    \sum_{n\geq 1}(1-\beta)^nv_{g,n}&=\frac{\tilde p}{p}\sum_{n\geq 1}\tilde{p}^{n-1}(1-\tilde{p})^n\tilde{v}_{g,n}\nonumber\\
    &=\frac{2^{g-1}\tilde{p}^{g}}{p}\left(\frac{1}{g}-\frac{2\tilde{p}}{g+1}\right).\label{eq:sum3_an}
    \end{align}
  Using the same arguments as the proof of Proposition~\ref{th:ancestralresistant} (i), that concludes the proof of Lemma~\ref{lem:TN_annexe} (i).
  
As previously, to find the density of $\tilde{T}_N$, it is sufficient to notice that the life time of each sensitive cells is distributed as exponential r.v. with parameter $b_0+d_0$, and that the chosen \textit{ancestral resistant} cell has $\tilde{G}_N$ ancestors. Thus, $\tilde{T}_N=\sum_{i=1}^{\tilde{G}_N}\mathcal{E}_i$, where $(\mathcal{E}_i)_{i\in \N}$ is a sequence of i.i.d. r.v. of exponential law with parameter $b_0+d_0$ and independent from $\tilde{G}_N$. That is, 
	\begin{align*}
	    f_{\tilde{T}_N}(t)&= \sum_{g=1}^\infty \P(\tilde{G}_N=g)f_{\Gamma(g,b_0+d_0)}(t)\\
	    &=\sum_{g=1}^\infty\frac{2^{g-1}}{p_N-\tilde{p}_N}\left[\frac{1}{g}((p_N)^{g}-(\tilde{p}_N)^{g})-\frac{2}{g+1}((p_N)^{g+1}-(\tilde{p}_N)^{g+1}) \right]\frac{t^{g-1}(b_0+d_0)^ge^{-t(b_0+d_0)}}{(g-1)!}\\
	    &=\frac{e^{-t(b_0+d_0)}}{p_N-\tilde p_N}\left[\frac{1}{2t}(e^{2t(b_0+d_0)p_N}-e^{2t(b_0+d_0)\tilde p_N})\right.\\
	    &\qquad\qquad\left. -\frac{1}{2t^2(b_0+d_0)}\left(e^{2p_Nt(b_0+d_0)}(2p_Nt(b_0+d_0)-1)-e^{2\tilde{p_N}t(b_0+d_0)}(2\tilde{p_N}t(b_0+d_0)-1)\right)\right],
	\end{align*}
	by noticing, for the last equality, that $\sum_{g=1}^\infty\frac{x^{g+1}}{(g-1)! (g+1)}=\sum_{g=1}^\infty\frac{x^{g+1}}{g!}-\frac{x^{g+1}}{(g+1)!}=e^x (x-1)+1$. And the last formula can be rewritten to conclude the proof of Lemma~\ref{lem:TN_annexe} (ii).
\end{proof}
\bibliographystyle{plain}
\bibliography{biblio}

\end{document}